\documentclass[a4paper,12pt,oneside,reqno]{amsart}
\usepackage{hyperref}
\usepackage[headinclude,DIV13]{typearea}
\areaset{15.1cm}{25.0cm}
\usepackage{amsfonts,amssymb,amsmath,amsthm,bbm}

\usepackage[latin1] {inputenc}
\usepackage{graphicx, psfrag}

\usepackage{subfigure}

\newtheorem{theorem}{Theorem}[section]
\newtheorem{lemma}[theorem]{Lemma}
\newtheorem{proposition}[theorem]{Proposition}

\theoremstyle{definition}

\theoremstyle{remark}
\newtheorem*{remark}{Remark}

\def\paragraph#1{\noindent \textbf{#1}}

\numberwithin{equation}{section}

\def\laweq{\overset{\mathrm{law}}=}

\def\supp{\mathop{\mathrm{supp}}\nolimits}

\def\supp{\mathop{\rm supp}\nolimits}

\def\d{\mathrm{d}}

\def\<{\langle}
\def\>{\rangle}

\def\e{\epsilon}

\def\g{\gamma}
\def\k{\kappa}
\def\l{\lambda}

\def\s{\sigma}

\def\o{\omega}
\def\D{\Delta}
\def\L{\Lambda}

\def\O{\Omega}
\def\S{\Sigma}

\def\del{\partial}
\def\R{{\Bbb R}}  
\def\N{{\Bbb N}}  
\def\P{{\Bbb P}}

\def\E{{\Bbb E}}

\let\cal=\mathcal

\def\BB{{\cal B}}

\def\EE{{\cal E}}
\def\FF{{\cal F}}
\def\GG{{\cal G}}

\def\LL{{\cal L}}

\def\NN{{\cal N}}

\def\TT{{\cal T}}

 \def \L {{\Lambda}}
 \def \k {{\kappa}}

 \def \s {{\sigma}}

 \def \D {{\Delta}}

 \def \g {{\gamma}}
 \def \l {{\lambda}}
 \def \d {{\delta}}
 
 \def \o {{\omega}}
 \def \O {{\Omega}}

 \def \del {{\partial}}

 \def \ba {\begin{array}}
 \def \ea {\end{array}}

 \def \eee{\hbox{\rm e}}
 \newcommand{\be}{\begin{equation}}
 \newcommand{\ee}{\end{equation}}

\newcommand{\bea}{\begin{eqnarray}}
 \newcommand{\eea}{\end{eqnarray}}
\def\TH(#1){\label{#1}}\def\thv(#1){\ref{#1}}
\def\Eq(#1){\label{#1}}\def\eqv(#1){(\ref{#1})}

\def\sfrac#1#2{{\textstyle{#1\over #2}}}

 \def \1{\mathbbm{1}}
\def\wt {\widetilde}
\def\wh{\widehat}

\def\ov {\overline }

\usepackage{subfigure}
\begin{document}


 \title[Variable speed branching Brownian motion]
{
Variable speed branching Brownian motion 1.
\\  Extremal processes in the weak correlation regime}
\author[A. Bovier]{Anton Bovier}
 \address{A. Bovier\\Institut f\"ur Angewandte Mathematik\\
Rheinische Friedrich-Wilhelms-Universit\"at\\ Endenicher Allee 60\\ 53115 Bonn, Germany }
\email{bovier@uni-bonn.de, \url{wt.iam.uni-bonn.de/bovier}}

\author[L. Hartung]{Lisa Hartung}
 \address{L. Hartung\\Institut f\"ur Angewandte Mathematik\\
Rheinische Friedrich-Wilhelms-Universit\"at\\ Endenicher Allee 60\\ 53115 Bonn, Germany}
\email{lhartung@uni-bonn.de, \url{wt.iam.uni-bonn.de/hartung}}


\date{\today}

 \begin{abstract}  
 We prove the convergence of the extremal processes for variable speed 
 branching Brownian motions where the  "speed functions", that describe the 
 time-inhomogeneous variance,  lie strictly below their 
 concave hull and satisfy a certain  weak regularity condition. These limiting 
 objects are universal in the sense that they only depend on the slope of 
 the speed function at $0$ and the final time  $t$. The proof is based on  
 previous results for two-speed BBM obtained in \cite{BovHar13}  and uses 
 Gaussian comparison arguments to extend these to the general case.

 \end{abstract}

\thanks{A.B. is partially supported through the German Research Foundation in 
the Collaborative Research Center 1060 "The Mathematics of Emergent Effects", 
the Priority Programme  1590 ``Probabilistic Structures in Evolution'',
the Hausdorff Center for Mathematics (HCM), and  the Cluster of Excellence ``ImmunoSensation'' at Bonn University.
L.H. is supported by the German Research Foundation in the Bonn International Graduate School in Mathematics (BIGS). 
}

 \maketitle


\section{Introduction}

Gaussian processes indexed by trees is a topic that received a lot of attention,  in particular in 
the context of spin glass theory (see e.g. \cite{BovStatMech, Tala1,Tala2, Pan-book})
through the so-called Generalised Random Energy Models (GREM), introduced and 
studied by
Derrida and Gardner \cite{D85,GD86a,GD86b}. Other contexts where such processes 
appeared are branching random walks 
(see e.g.\cite{Bramson-rw,Shi,ZeitouniLN}) and branching Brownian motion  
(see e.g. \cite{Moyal,McKean,B_M,B_C, DerriSpohn88}).

One of the issues of interest in this context is to
understand the structure of the extremal processes that arise in these models 
in the limit when the size of the tree tends to infinity. A Gaussian process on a tree is 
characterised fully by the tree and by its covariance, which in the models we are interested in is 
a function of the genealogical distance on the tree.
In the classical models of branching random walk and branching Brownian motion, 
the covariance is  a linear function of
the tree-distance. 
 In the context of the GREM, the tree is  
a binary tree with $N$ levels; another popular tree is a supercritical 
Galton-Watson tree (see, e.g. \cite{athreya-ney}). These 
models generalise branching Brownian motion and 
 were  first introduced, to our knowledge,  by Derrida and Spohn \cite{DerriSpohn88}.

In this paper we focus on this latter class of models. 
They  can be constructed as follows. On some abstract probability space $(\O,\FF,\P)$, 
define a supercritical Galton-Watson (GW) 
 tree. The offspring distribution, $\{p_k\}_{k\in \N}$, is normalised for convenience  such that  
 $\sum_{i=1}^\infty p_k=1$,
$\sum_{k=1}^\infty k p_k=2$ , and the second moment,  $K=\sum_{k=1}^\infty k(k-1)p_k$ is 
assumed finite. 
We fix a time horizon $t>0$. 
We denote the number of individuals ("leaves") of the tree at time $t$ by 
$n(t)$ and label the leaves at time $t$ by $i_1(t), i_2(t),\dots, i_{n(t)}(t)$. 
For given $t$ and for  $s\leq t$, it is convenient to let $i_k(s)$ denote the
ancestor of particle $i_k(t)$ at time $s$. Of course, in general there will be several indices
 $k,\ell$ such that $i_k(s)=i_\ell(s)$. The time of the most recent common ancestor of $i_k(t)$ 
 and $i_\ell(s)$ is given, for $s, r\leq t$, by
\be\Eq(gw.1)
d(i_k(r),i_\ell(s))\equiv \sup\{u\leq s\land r: i_k(u)=i_\ell(u)\}.
\ee
We denote by $\FF^{\hbox{\tiny tree}}_s, s\in \R_+$ the $\s$-algebra generated by the Galton-Watson 
process up to time $s$. On the same probability space we will now construct,
for given $t$, and for any realisation of the GW tree, a Gaussian process as follows.

 Let $A:[0,1]\rightarrow [0,1]$ be a right-continuous non-decreasing function.
 We define a Gaussian process, $x$, labelled by the tree (up to time $t$),
 i.e. by $\{i_k(s)\}_{1\leq k\leq n(t)}^{ 0\leq s\leq t}$, with covariance, 
 for $0\leq s,r\leq t$ and $k,\ell \leq n(t)$ 
\be\Eq(variance.2)
\E \left[x_k(s)x_\ell(r)\right] = tA\left(t^{-1}d(i_k(r),i_\ell(s)\right).
\ee

The existence of such a process is shown easily through a construction as time 
changed branching Brownian motion. Note first that, in the case when $A(x)=x$, this process is
standard branching Brownian motion \cite{Moyal,Skorohod64}. For general  $A$, 
the models 
can by constructed from \emph{time changed} Brownian motion as follows.
Let 
\be\Eq(intspeed.2)
\S^2(s)= t A(s/t).
\ee
Note that $\S^2$ is almost everywhere differentiable and denote by 
$\s^2(x)$ its derivative wherever it exists. 
Define the process  $\{B^\S_s\}_{0\leq s\leq t}$ 
 on $[0,t]$ as time change of ordinary 
Brownian motion, $B$, via
\be\Eq(speedy.1)
B^\S_s=B_{\S^2(s)}.
\ee 
Branching Brownian motion with speed function $\S^2$ is constructed like ordinary Brownian 
motion, except that if a particle splits at some time $s<t$, then the offspring particles perform 
variable speed Brownian motions with speed function $\S^2$, i.e. they  are independent copies
of  $\{B^\S_r-B^\S_s\}_{t\geq r\geq s}$, all starting at the position of the parent particle at time 
$s$.
We  refer to these processes as \emph{variable speed branching Brownian motion}. 
This class of processes, labelled by the different choices of functions $A$, provides an 
interesting set of examples to study the possible limiting extremal processes for 
correlated 
random variables. The ultimate goal will be to describe the extremal processes in 
dependence 
on the function $A$.

\begin{remark} Strictly speaking, we are not talking about a single stochastic process, but 
about a family, $\{x^t_k(s),k\leq n(s)\}_{s\leq t}^{t\in \R_+}$, of processes with finite time horizon, 
indexed by that horizon, $t$. That dependence on $t$ is usually  not made explicit in order not to 
overburden the notation.
\end{remark}

Branching Brownian motion has received a lot of attention over the last decades, with a strong 
focus on the properties of extremal particles. We mention the seminal contributions of McKean 
\cite{McKean}, Bramson \cite{B_M, B_C}, Lalley and Sellke \cite{LS}, and Chauvin and Rouault 
\cite{chauvin88,chauvin90} on the connection to the Fisher-Kolmogorov-Petrovsky-Piscounov 
(F-KPP) equation \cite{fisher37, kpp} and on the distribution of the rescaled maximum. In 
recent years, there has been a revival of interest in BBM with numerous contributions, including 
the construction of the full extremal process by A\"\i d\'ekon et al. \cite{abbs} and 
Arguin et al. \cite{ABK_E}.
For a review of these developments see, e.g., the recent survey by Gou\'er\'e \cite{gouere} or 
the lecture notes \cite{BoPrag13}.
Variable speed branching Brownian motion (as well as random walk) has    recently been 
investigated by 
Fang and Zeitouni \cite{FZ_RW,FZ_BM}, Maillard and Zeitouni \cite{MZ}, Mallein \cite{Mallein}, and the present authors \cite{BovHar13}.

Naturally, the same construction can be done for any other family of trees. It is widely believed 
(see \cite{ZeitouniLN}) that the resulting structures are very similar, with only details depending 
on the underlying tree model. 
More importantly, it is believed that the extremal structure in more general Gaussian 
processes, such as mean field spin glasses \cite{BoKi1,BoKi2}  or the Gaussian free field 
\cite{ZeitouniLN} are of the same type; considerable progress in this direction has been made 
recently by  Bramson, Ding, and Zeitouni \cite{BraDiZei} and by Biskup and Louidor \cite{BisLou13}.

We are interested in understanding the nature of the extremes of our processes in dependence 
on the properties of the covariance functions $A$.
The case when $A$ is a step function with finitely many steps corresponds 
to Derrida's 
GREMs \cite{GD86b,BK1}, the only difference being that the deterministic 
binary tree of 
the GREM is replaced by a Galton-Watson tree. 
It is very easy to treat this case. 

 The case when $A$ is arbitrary has been dubbed 
CREM in \cite{BK2} (and treated for binary regular trees). In that case the leading order of the 
maximum was obtained, as well as the genealogical description of the 
Gibbs measures; this 
analysis carries over mutando mutandis to the analogous  BBM situation. 
The finer analysis of 
the extremes is, however, much more subtle and in general still open.
Fang and Zeitouni  \cite{FZ_BM} have obtained the order of the  corrections (namely $t^{1/3}$) 
in the case when $A$ is strictly concave and continuous. These  
corrections come naturally from 
the probability of a Brownian bridge to stay away from a curved
line, which was earlier analysed by Ferrari and Spohn \cite{FerSpo}.
There are, however, no results on the extremal process or the law of the 
maximum.  

Another rather tractable situation occurs when $A$ is a piecewise linear 
function.  
The simplest case here corresponds
to choosing a speed that takes just two values, i.e.
\be\Eq(speedy.2)
\s^2(s)=\begin{cases}
 \s^2_1, &\,\hbox{for}\,\,0\leq s< tb,\\
 \s^2_2, &\,\hbox{for}\,\,bt\leq s\leq t,
 \end{cases}
 \ee
 with $\s_1^2b+\s_2^2(1-b)=1$. In this case, 
 Fang and Zeitouni  \cite{FZ_BM}
have  obtained the correct order of the logarithmic  
corrections.  
  This case was fully analysed in a recent paper of ours 
  \cite{BovHar13}, where we provide the construction of the 
  extremal processes. 
  
In the present paper, we present the full picture in the case where  
$A(x)<x$ for all $x\in (0,1)$, and the slopes of $A$ at $0$ and at $1$ are different from $1$. We 
show that 
there is a large degree of universality in that the limiting extremal processes are those that 
emerged in the two-speed case, and that they depend only on the slopes of $A$ at $0$ and at 
$1$. 

The critical cases,  $A(x)\leq x$, involve, besides the well-understood standard BBM, a 
number of different situations  that can be quite 
tricky, and we postpone this analysis to a forthcoming publication.

\subsection{Results}

We  need some mild technical assumptions on the covariance function.  Let $A:
[0,1]\rightarrow [0,1]$ be a right-continuous, non-decreasing function that  satisfies the following 
three conditions:
\begin{itemize}
\item[(A1)] For all $x\in (0,1)$: $A(x)<x$, $A(0)=0$ and $A(1)=1$.
\item[(A2)]There exists $\d_b>0$ and functions $\overline{B}(x)$, $\underline{B}(x):
[0,1]\rightarrow [0,1]$ that are twice differentiable in $[0,\d_b]$ with bounded second 
derivatives, such that
\be\Eq(ass.2)
\underline{B}(x)\leq A(x)\leq  \overline{B}(x),\quad \forall x\in[0,\d_b]
\ee
with $\overline{B}'(0)=\underline{B}'(0)\equiv A'(0)$.
\item[(A3)]There exists $\d_e>0$ and functions $\overline{C}(x)$, $\underline{C}(x):
[0,1]\rightarrow [0,1]$ that are twice differentiable in $[1-\d_e,1]$ with bounded second 
derivatives, such that
\be\Eq(ass.3)
\underline{C}(x)\leq A(x)\leq  \overline{C}(x),\quad \forall x\in[1-\d_e,1]
\ee
with $\overline{C}'(1)=\underline{C}'(1)\equiv A'(1)$. The case $A'(1)=+\infty$ is allowed. 
This is to be understood in the sense that,  for all $\rho <\infty$, there exists $
\varepsilon>0$ such that, for all 
 $x\in [1-\varepsilon,1]$, $A(x) \leq 1-\rho(1-x)$. 
\end{itemize}

For standard BBM, $\bar x(t)$, recall that 
Bramson \cite{B_M} and Lalley and Sellke \cite{LS} have shown that
\be\Eq(old.1)
\lim_{t\uparrow \infty} \P\left(\max_{k\leq n(t)}\bar x_k(t)-m(t) \leq y\right)
=\o(x) =\E\left[ \eee^{-CZ\eee^{-\sqrt 2 y}}\right],
\ee
where $m(t)\equiv \sqrt 2 t-\frac{3}{2\sqrt 2} \log t$, $Z$ is a random variable, the limit of the so 
called \emph{derivative martingale}, and $C$ is a constant.

In \cite{ABK_E} (see also 
\cite{abbs}  for a different proof) it was shown that the extremal process, 
\be 
\Eq(old.2)
\lim_{t\uparrow\infty} \wt\EE_t\equiv \lim_{t\uparrow\infty}\sum_{k=1}^{n(t)} \d_{\bar x_k(t)-m(t)}
=\wt\EE,
\ee
exists in law, and $\wt\EE$ is of the form 
\be\Eq(old.3)
\wt\EE =\sum_{k,j}\d_{\eta_k+\D^{(k)}_j},
\ee
where $\eta_k$ is the $k$-th atom of a Cox process 
\cite{Cox55}) directed by the random 
measure $CZ \eee^{-\sqrt 2 y}dy$, with $C$ and $Z$ as before. 
$\D^{(k)}_i$ are the atoms of independent and identically distributed  point processes 
$\D^{(k)}$, which are the limits in law
of 
\be
\sum_{j\leq n(t) }\d_{\tilde x_i(t)-\max_{j\leq n(t)}\tilde x_j(t)},
\Eq(old.4)
\ee 
where $\tilde x(t)$ is BBM conditioned on the event $\max_{j\leq n(t)} \tilde x_j(t)\geq \sqrt 2 t$.

The main result of the present paper is the following theorem.

\begin{theorem}\TH(maintheo) Assume that $A:[0,1]\to[0,1]$ satisfies (A1)-(A3). 
Let $A'(0) = \s_b^2<1$ and  $A'(1)=\s_e^2>1$. Let 
$\tilde m(t) =\sqrt 2 t-\frac 1{2\sqrt 2} \log t$.  
Then
there is a constant $\wt C(\s_e)$ depending only on $\s_e$ and a random variable $Y_{\s_b}$ 
depending only on 
$\s_b$ such that 
\begin{itemize}
\item[(i)]
\be\Eq(maintheo.1)
\lim_{t\uparrow \infty}\P\left(\max_{1\leq i\leq n(t)}x_i(t)-\tilde m(t) \leq x\right) = \E\left[ \eee^{- \wt 
C(\s_e)Y_{\s_b} \eee^{-\sqrt 2 x}}\right].
\ee
\item[(ii)] The point process
\be\Eq(maintheo.2)
\sum_{k\leq n(t)} \d_{x_k(t)-\tilde m(t)}\to \EE_{\s_b,\s_e}=
\sum_{i,j} \d_{p_i+\s_e\L^{(i)}_j},
\ee
as $t\uparrow\infty$, in law, where the $p_i$ are the atoms of a Cox 
process on $\R$ 
directed by the random measure $\wt C(\s_e)Y_{\s_b}\eee^{-\sqrt 2x}dx$, and the 
$\L^{(i)}$ are the limits of the processes as in 
\eqv(old.4), but conditioned 
 on the event $\{\max_k \tilde x_k(t)\geq \sqrt 2 \s_et\}$.
 \item[(iii)]  
 If $A'(1)=\infty$, then $\wt C(\infty)=1/\sqrt {4\pi}$, and $\L^{(i)}=\d_0$, 
 i.e. the limiting process is a Cox process.
\end{itemize}
The random variable $Y_{\s_b}$ is the limit of the uniformly integrable martingale
\be\Eq(mckean.0)
Y_{\s_b}(s)=\sum_{i=1}^{n(s)}\eee^{-s(1+\s_b^2)+\sqrt{2}\s_b\bar x_i(s)},
\ee
where $\bar x_i(s)$ is standard branching Brownian motion.
\end{theorem}

\begin{remark} In Theorem 7.7 of \cite{BovHar13} the  constant $\wt C(\s_e)$ is characterised by the tail behaviour of 
 solutions to the F-KPP equation, namely
 \be\Eq(constant.1)
\wt C(\s_e)\equiv  \s_e\lim_{t\to\infty}e^{\sqrt{2}x}e^{x^2/2t}t^{1/2}u(t,x+\sqrt{2}t),
\ee
where $x\equiv \sqrt2(\s_e-1)t$,  and $u$ solves the F-KPP equation 
\be\Eq(fkpp)
\del_t u(t,x)=\frac 12 \del_x^2u(t,x) +(1-u(t,x))-\sum_{k=1}^\infty p_k (1-u(t,x))^k,
\ee
with initial condition $u(0,x)=\1_{x\leq 0}$. 
\end{remark}

\begin{remark} The special case of Theorem \thv(maintheo) when $A$ consists of two linear 
segments  was 
obtained in \cite{BovHar13}. Theorem \thv(maintheo) shows that the limiting objects under 
conditions $(A1)-(A3)$ are universal and depend only on the slopes of the covariance function 
$A$ at $0$ and at $1$. This could have been guessed,
but the rigorous proof turns out to be quite involved.  
Note that $\s_e=\infty$ is allowed. In that 
case  the extremal  process is just a mixture of Poisson point processes. 
If $\s_b=0$, then $Y_{\s_b}$ is just an exponential random variable of mean $1$. 
We call $(Y_{\s_b}(s))_{s\in \R_+}$ the \emph{McKean martingale}.
\end{remark}

 \subsection{Outline of the proof}
 The proof of Theorem \thv(maintheo) is based on the corresponding result 
 obtained in 
 \cite{BovHar13} for the case of two speeds, and  on a Gaussian 
 comparison method. We start 
 by showing the localisation of paths, namely that the paths of all particles 
 that reach a hight of 
 order $\tilde m(t)$ at time $t$ has to lie within a certain tube. Next, we
  show tightness of the  extremal process. 
 
 The remainder of the paper is then concerned with proving the  
 convergence of the finite dimensional
 distributions through  Laplace transforms. We introduce 
 auxiliary two speed BBM's  
whose covariance functions approximate $A$ well around $0$ 
and $1$. Moreover 
  we choose them in such a way that their covariance functions lie above 
  respectively below $A$ in a neighbourhood of $0$ and $1$ 
  (see Figure \ref{figA}).
 
We then use Gaussian comparison methods to compare the Laplace 
transforms. The 
Gaussian comparison comes in three main steps. In a first step we 
introduce the usual interpolating 
process and introduce a localisation  condition on its paths. 
In a second step we justify a certain integration 
by parts formula, that is adapted to our setting. Finally, the resulting 
quantities are decomposed 
into a part with controlled  sign and a part that converges to zero.

\begin{figure}[hbtp]
 \includegraphics[width=10cm]{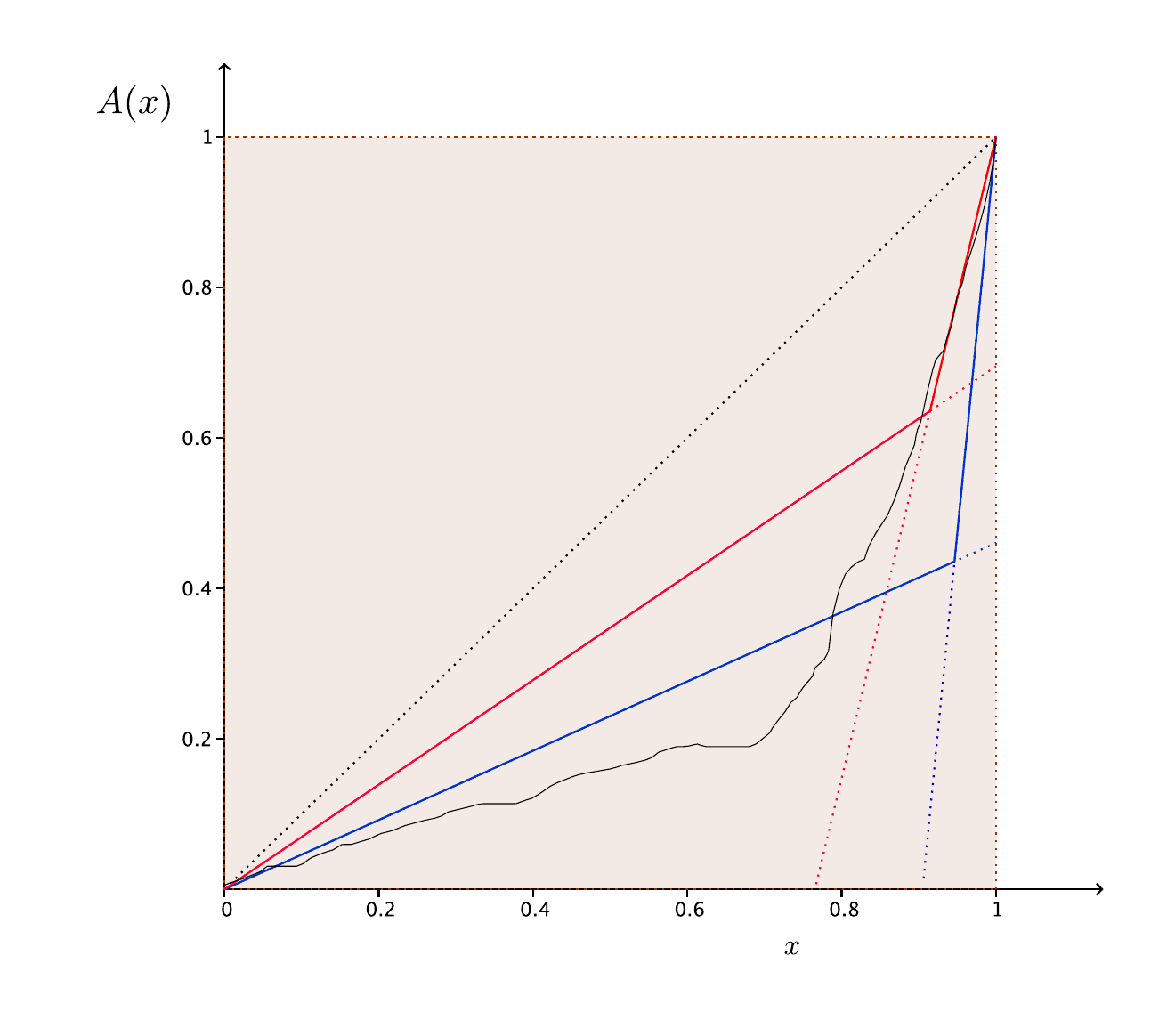}
 \caption{Gaussian Comparison: The extremal process of BBM with covariance $A$ 
 (black curve) is compared to process with covariances functions 
  $\overline A$ (red curve), respectively $\underline A$  (blue, curve).}
 \label{figA}
\end{figure}
 
\noindent\textbf{Acknowledgements.} We thank an anonymous 
 referee for a very careful reading of our paper and for numerous valuable suggestions.

\section{Localization of paths}

In this section we show where the paths of particles that are extreme at time $t$
are localised. This is essentially inherited from properties of the standard Brownian bridge.  For a given speed function $\S^2$, and a subinterval $I\subset [0,t]$, define the following events on the space of paths, $X:\R_+\to\R$,
\be\Eq(path.2)
\TT_{t,I,\Sigma^2}^\g=\left\{X \Big\vert \forall s: s\in I:\left\vert X(s)-\sfrac{\Sigma^2(s)}{t}X(t)\right\vert<(\Sigma^2(s)\wedge (t-\Sigma^2(s)))^\gamma \right\}.
\ee

\begin{proposition}\TH(prop.loc.1)
Let $x$ denote  variable speed BBM with covariance function $A$. For $0\leq r< t$, set 
  $I_r\equiv \{s: \S^2(s)\in [r,t-r]\}$.
 For any  $\g >\frac{1}{2}$ and 
for all  $d\in\R$,  for all $\e>0$, there exists $r_0<\infty$ such that, for $r>r_0$ and for all $t>3r$,
\be\Eq(path.1)
\P\left(\exists k\leq n(t):\left\{ x_k(t)>\tilde m(t)+d\right\}\wedge \left\{x_k\not \in \TT_{t,I_r, \Sigma^2}^\g\right\}\right)<\e.
\ee
\end{proposition} 

\noindent To prove  Proposition $\thv(prop.loc.1)$ we need the following  lemma on Brownian bridges (see \cite{B_C}).

\begin{lemma}\TH(BB.basic)
Let $\g>\frac{1}{2}$. Let $\xi$ be a Brownian bridge from $0$ to $0$ in time $t$. Then, for all $\e>0$,  there exists $r_0<\infty$ such that, for $r>r_0$ and for all $t>3r$,
\be\Eq(path.3.0)
\P\left(\exists s\in [r,t-r] :|\xi(s)|>(s\wedge(t-s))^\g\right)<\e.
\ee
More precisely,
\be\Eq(path.3.1)
\P\left(\exists s\in [r,t-r]:|\xi(s)|>(s\wedge(t-s))^\g\right)<8\sum_{k=\lfloor r\rfloor}^\infty k^{\frac{1}{2}-\g}\eee^{-k^{2\g-1}/2}.
\ee
\end{lemma}
\begin{proof}
The probability in \eqv(path.3.0) is bounded from above by
\bea\Eq(path.4)
&&\sum_{k=\lceil r\rceil}^{\lceil t-r\rceil}\P\left(\exists s\in [k-1,k]:|\xi(s)|>(s\wedge(t-s))^\g\right)\nonumber\\
\leq && 2\sum_{k=\lceil r\rceil}^{\lceil t/2\rceil}\P\left(\exists s\in [k-1,k]:|\xi(s)|>(s\wedge(t-s))^\g\right),
\eea
by the reflection principle for the Brownian bridge. This is  bounded from above by
\be\Eq(path.5)
2\sum_{k=\lceil r\rceil}^{\lceil t/2\rceil}\P\left(\exists s\in [0,k]:|\xi(s)|>(k-1)^\gamma\right). 
\ee
Using the bound of Lemma 2.2 (b) of \cite{B_C} we have  
\be\Eq(path.6)
P\left(\exists s\in [0,k]:|\xi(s)|>(k-1)^\gamma\right)\leq 4(k-1)^{\frac{1}{2}-\g}\eee^{-(k-1)^{2\g-1}/2}.
\ee
Using this bound for each summand in \eqv(path.5) we obtain \eqv(path.3.1). Since the sum on the right-hand side of \eqv(path.3.1) is finite \eqv(path.3.0) follows.
\end{proof}

\begin{proof}[Proof of Proposition \thv(prop.loc.1)]
 Using a first moment method, the probability in \eqv(path.1) is bounded from above by
 \be\Eq(path.7)
 \eee^t\P\left(B_{\Sigma^2(t)}>\tilde m(t)+d , B_{\Sigma^2(\cdot)}\not\in\TT_{t,I_r,\Sigma^2}^\g\right).
 \ee
  Since   
  $\Sigma^2(s)$ is an non-decreasing function on $[0,t]$ with $\Sigma^2(t)=t$,
  the expression in  \eqv(path.7) is bounded from above by
 \be\Eq(path.8)
 \eee^t\P\left(\left\{B_t>\tilde m(t)+d\right\}\land \left\{\exists s\in [r,t-r]:\left\vert B_s-\frac{s}{t}B_t\right\vert>(s\wedge (t-s))^\g\right\}\right).
 \ee
Now,   $\xi(s)\equiv B_s-\frac{s}{t}B_t$ is the
 Brownian bridge from $0$ to $0$ in time $t$, and it is 
well known (see e.g. Lemma 2.1 in \cite{B_C}) that  $\xi(s)$ is  independent of $B_t$, 
for all $s\in [0,t]$. Therefore,   \eqv(path.8) is equal to
 \be\Eq(path.9)
 \eee^t\P\left(B_t>\tilde m(t)+d\right)\P\left(\exists s\in [r,t-r]:\left\vert \xi(s)\right\vert>(s\wedge (t-s))^\g\right).
 \ee
Using the standard Gaussian tail bound, 
\be\Eq(gaussian)
\int_{u}^\infty \eee^{-x^2/2}dx\leq u^{-1}\eee^{-u^2/2},\quad\text{for}\,\, u>0,
\ee
we have
 \bea\Eq(path.10)
 \eee^t\P\left(B_t>\tilde m(t)+d\right)
 &\leq& \eee^t \frac{\sqrt{t}}{\sqrt{2\pi}(\tilde m(t)+d)}\eee^{-(\tilde m(t)+d)^2/2t}\nonumber\\
 &\leq& \frac{t}{\sqrt{2\pi}(\tilde m(t)+d)}\eee^{-\sqrt{2}d}
 \leq M,
 \eea
 for some constant $M$ (depending on $d$), if $t$ is large enough.
 By Lemma  
 \thv(BB.basic) we can find  $r_0$ large enough such that for all $r\geq r_0$ and $t>3r$,
 \be\Eq(path.11)
 \P\left(\exists s\in [r,t-r]:\left\vert \xi(s)\right\vert>(s\wedge(t-s))^\g\right)<\e/M.
 \ee
 The bounds  \eqv(path.10) and \eqv(path.11) imply  that   \eqv(path.9) is smaller than
  $\e$.
\end{proof}

 \section{Proof of Theorem \thv(maintheo)}

In this section we  prove Theorem \thv(maintheo) assuming Proposition \thv(prop.number) below, whose proof will be postponed to the following two sections.

\begin{proof}[Proof of Theorem \thv(maintheo)]
We show the convergence of the extremal process
\be\Eq(tight.1)
\EE_t=\sum_{k\leq n(t)}\d_{x_k(t)-\tilde m(t)}
\ee
by showing the convergence of the finite dimensional distributions and tightness. Tightness of $(\EE_t)_{t\geq 0}$ is implied by the following bound on the 
number of particles above a level $d$ (see \cite{Res}, Lemma 3.20).   
\begin{proposition}\TH(prop.tight)
For any $d\in\R$ and $\e>0$, there exists $N=N(\e,d)$ such that, for all $t$ large enough,
\be\Eq(tight.6)
\P\left(\EE_t[d,\infty)\geq N\right)<\e.
\ee
\end{proposition}
\begin{proof}
 By a first order Chebyshev inequality, for all $t$ large enough, 
 \be\Eq(tight.2)
 \P\left(\EE_t[d,\infty)\geq N\right)\leq \frac{1}{N}\eee^t\P\left(B_t>\tilde m(t)+d\right)\leq  \frac{M}{N} 
 \ee
 by \eqv(path.10), where $M>0$ is a constant that depends on $d$. Choosing $N>M/\e$  yields Proposition \thv(prop.tight).
\end{proof}

 To show the convergence of the finite dimensional distributions define, for $u\in\R$,
 \be\Eq(tight.3)
 \NN_{u}(t)=\sum_{i=1}^{n(t)}\1_{x_i(t)-\tilde m(t)>u},
 \ee
 that counts the number of points that lie above $u$. Moreover, we define the corresponding quantity for the process  $\EE_{\s_b,\s_e}$ (defined in 
 \eqv(maintheo.2)),
  \be\Eq(tight.4)
 \NN_u=\sum_{i,j}\1_{p_i+\s_e\L_j^{(i)}>u}.
 \ee
Observe that, in particular,
\be\Eq(tight.max)
\P\left(\max_{1\leq i\leq n(t)}x_i(t)-\tilde m(t)\leq u\right)=\P\left(\NN_{u}(t)=0\right).
\ee
The key step in the proof of Theorem \thv(maintheo) is the following proposition, that asserts the convergence of the finite dimensional distributions of the  process $\EE_t$.

 \begin{proposition}\TH(prop.number)
 For all $k\in\N$ and $u_1,\dots,u_k\in\R$
 \be\Eq(tight.5)
 \{\NN_{u_1}(t),\dots,\NN_{u_k}(t)\}\stackrel{d}{\to} \{\NN_{u_1},\dots,\NN_{u_k}\}
 \ee
 as $t\uparrow \infty$.
 \end{proposition}
 
 \noindent The proof of this proposition  will be postponed to the following sections.

Assuming the proposition, we can now conclude the proof of the theorem.
  The distribution of
$\{\NN_{u_1}(t),\dots,\NN_{u_k}(t)\}$ for all $k\in \N,u_1,\dots,u_k\in\R$ characterise the finite dimensional distributions of the point process $\EE_t$ since
the class of sets $\{(u,\infty), u\in \R\}$ form a $\Pi$-system  that generates  
$\BB(\R)$. Hence \eqv(tight.5) implies the convergence of the finite dimensional distributions of $\EE_t$ (see, e.g.,  Proposition 3.4 in \cite{Res}). 

Combining this observation with Propositions \thv(prop.tight), we obtain Assertion
(ii)  of 
Theorem \thv(maintheo). Assertion (i) follows immediately from Eq. 
\eqv(tight.max).

To prove Assertion (iii), we need to show that, as $\s^2_e\uparrow \infty$, it holds that $\wt C(\s_e) \uparrow 1/\sqrt {4\pi}$ and the processes $\L^{(i)}$ converge to the trivial process $\d_0$. 
Then, 
\be\Eq(infinity.1)
\EE_{\s_b,\infty} =\sum_{i}\d_{p_i},
\ee
where $(p_i,i\in \N)$ are the points of a Cox process directed by the  random 
 measure\hfill\eject $\frac{1}{\sqrt{4\pi}}Y_{\s_b}\eee^{-\sqrt{2}x}dx$. 

\begin{lemma}\TH(lem.infinity)
The point process $\EE_{\s_b,\s_e}$ converges in law, as $\s_e\uparrow \infty$, to the point process $\EE_{\s_b,\infty}$.

\end{lemma}

\begin{proof}
 The proof of Lemma \thv(lem.infinity) is based on a result concerning the cluster processes
 $\L^{(i)}$. 
 We write $\L_{\s_e}$ for a single copy of these processes and add the subscript to make the dependence on the parameter $\s_e$ explicit. 
 We recall from \cite{BovHar13} that the process $\L_{\s_e}$ is constructed as follows. 
Define the processes 
 $\ov{\EE}_{\s_e}$  as the  limits of the point processes 
 \be\Eq(lem33.1)
 \ov{\EE}_{\s_e}^t\equiv 
 \sum_{k=1}^{n(t)} 
 \d_{x_{k}(t)-\sqrt 2\s_e t}
 \ee 
 where $x$ is standard BBM at time $t$ conditioned on the event 
 $\{\max_{k\leq n(t)} x_k(t)>\sqrt 2\s_e t\}$. 
 We  show here that, as $\s_e$ tends to infinity, the processes $\ov{\EE}_{\s_e}$
 converge to a point process consisting of a single atom at $0$.  More precisely, we 
 show that 
\be\Eq(palm.1)
\lim_{\s_e\uparrow\infty}\lim_{t\uparrow \infty}\P\left(\ov\EE^t_{\s_e}([-R,\infty))>1\big|
\max_{k\leq n(t)} x_k(t)>\sqrt 2\s_e t\right)=0.
\ee
Now, 
\bea\Eq(palm.2)
&&\P\left(\ov\EE^t_{\s_e}([-R,\infty))>1\big|
\max_{k\leq n(t)} x_k(t)>\sqrt 2\s_e t\right)\\\nonumber
&&\leq 
\P\left(\supp\ov\EE^t_{\s_e} \cap [0,\infty)\neq \emptyset  \land \ov\EE^t_{\s_e}([-R,\infty))>1\big|
\max_{k\leq n(t)} x_k(t)>\sqrt 2\s_e t\right)\\\nonumber
&&\leq \int_0^\infty \P\left(\supp\ov\EE^t_{\s_e} \cap dy\neq \emptyset \land 
\ov\EE^t_{\s_e}([-R,\infty))>1\big|
\max_{k\leq n(t)} x_k(t)>\sqrt 2\s_e t\right)\\\nonumber
&&= \int_0^\infty \P\left(\supp\ov\EE^t_{\s_e} \cap dy\neq \emptyset
\big| \max_{k\leq n(t)} x_k(t)>\sqrt 2\s_e t \right)\\\nonumber
&&\qquad\times\P\left( \ov\EE^t_{\s_e}([-R,\infty))>1\big|\supp\ov\EE^t_{\s_e} \cap dy\neq \emptyset\right).
\eea
But $\P\left( \cdot \big|\supp\ov\EE^t_{\s_e} \cap dy\neq \emptyset\right)\equiv
 P_{t,y+\sqrt 2\s_e}(\cdot)$ is the \emph{Palm measure} on BBM, i.e. the conditional law 
 of BBM given that there  is a particle at time $t$ in  $dy$ 
 (see Kallenberg \cite[Theorem 12.8]{Kallen}. 
 Chauvin, Rouault, and Wakolbinger 
 \cite[Theorem 2]{ChRoWa91} describe the tree under the   Palm measure $P_{t,z}$ as 
 follows.
 Pick one particle at time $t$ at the location 
 $z$. Then pick a  
 \emph{spine}, $Y$,  which is a Brownian bridge from $0$ to $z$ in time $t$. 
 Next pick a Poisson point process $\pi$  on $[0, t]$ with intensity  $2$.
 For each point $p\in \pi$ start a random number 
  $\nu_p$ of independent branching Brownian motions $(\BB^{Y(p),i},i\leq \nu_p)$ 
  starting at $Y(p)$.
  The law of $\nu$ is given by the size biased distribution, $\P(\nu_p=k-1)\sim \frac{kp_k}
  {2}$. See Figure \ref{figure.1}.
 Now let $z= \sqrt 2\s_e t+y $ for $y\geq 0$. Under  the Palm measure, the point process 
$ \ov{\EE}_{\s_e}(t)$ then takes the form
  \be
  \ov{\EE}_{\s_e}(t) \laweq \d_y+\sum_{p\in\pi,i<\nu_p}
 \sum_{j=1}^{n_{Y(p),i}(p)}\d_{\BB^{Y(p),i}_j(t-p)-\sqrt 2\s_et}.
 \ee
 \begin{figure}
 \flushleft \hspace{2cm}
\subfigure  {  \includegraphics[width=0.3\textwidth]{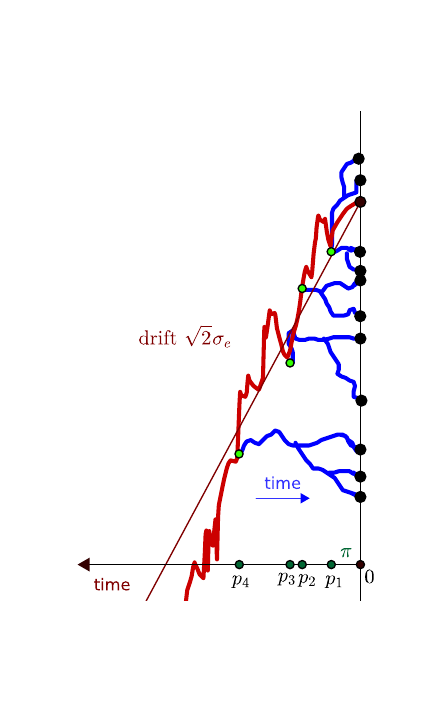}}\hfill
\subfigure  { \hspace{-7cm}\includegraphics[width=0.25\textwidth ]{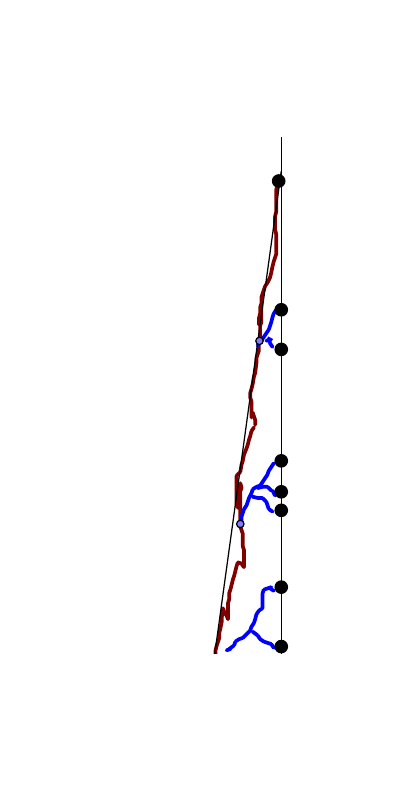}}
\caption{The cluster process seen from infinity for $\s_e$ small (left) and $\s_e$ very large (right)}
\label{figure.1}
\end{figure}
  Since, for $1>\g>1/2$,
 \be\Eq(infinity.8)
\lim_{\s_e\uparrow \infty}\lim_{t\uparrow\infty}\P\left(\forall s\geq \s_e^{-1/2} :Y(t-s)-y+\sqrt{2}\s_e
 s\in [-(\s_e s)^\g, 
(\s_e s)^\g] \right)=1,
 \ee
 if we define the set 
 \be\Eq(infinity.8.1)
 \GG^t_{\s_e}\equiv\left\{Y: 
  \forall {t\geq s\geq  \s_e^{-1/2}} :Y(t-s) -y+\sqrt{2}\s_es\in [-(\s_e s)^\g, 
(\s_e s)^\g]\right\},
\ee
 it will suffice to show that,  for all $R\in \R_+$,  
  \be\Eq(infinity.5)
\lim_{\s_e\uparrow\infty} \lim_{t\uparrow \infty}\P\left(\exists {p\in\pi,i<\nu_p,j}:{\BB^{Y(p),i}_j(t-p)\geq  y-R} \land
Y\in\GG^t_{\s_e} \right)=0.
 \ee 
 The probability in \eqv(infinity.5) is  bounded by
 \bea\Eq(infinity.5.1)
  &&\P\left(\exists {p\in\pi,i\leq\nu_p,j}:{\BB^{Y(p),i}_j(t-p)\geq y -R)}\land
Y\in\GG^t_{\s_e}\right) \\\nonumber
 &\leq& \E\left[\int_0^t \sum_{i=1}^{\nu_p}  \1_{\BB^{Y(p),i}_{j}(t-p) >y-R}
 \1_{Y\in\GG_{\s_e}}
 \pi(dp)\right]\\\nonumber
 &\leq& \E\left[\int_0^t   \E\left[\sum_{i=1}^{\nu_p}  
 \1_{\max_j\BB^{Y(p),i}_{j}(t-p) \geq y-R}\1_{Y\in\GG^t_{\s_e}}
 \big|\FF^\pi\right]\pi(dp)\right]\\\nonumber  
 &\leq& \int_0^t 2K  \P\left(\max_j\BB^{Y(t-s)}_j \geq y-R\land
Y\in\GG^t_{\s_e}\right)ds.
 \eea
 Here we used the independence of the offspring BBM and that the conditional 
 probability  given the $\s$-algebra  $\FF^\pi$ generated by the Poisson process $\pi$
 appearing in the integral over $\pi$ depends only on $p$. 
 For the integral over $s$ up to $1/\s_\e^{1/2}$, we just bound the integrand by $2K$.
 For larger values, we use the localisation provided by the condition that 
 $Y\in \GG_{\s_e}$, to get that the right hand side of \eqv(infinity.5.1) is not larger than
 \bea\Eq(infinity.6)
  2K\int_0^{ \s_e^{-1/2}} ds
  + 2K\int_{\s_e^{-1/2}}^{t}\eee^s\P( B(s)>-R+\sqrt2 \s_es - (\s_es)^\g)ds.
 \eea
 \eqv(infinity.6) is by \eqv(gaussian) bounded from above by
 \be\Eq(infinity.7)
  2K\s_e^{-1/2} 
 + 2K\int_{\s_e^{-1/2}}^{\infty}\eee^{(1-\s_e^2)s+\sqrt 2\s_e(R+(\s_e s)^{\g})}ds.
 \ee
 From this it follows that \eqv(infinity.7) (which does no longer depend on $t$) converges to zero, as  $\s_e\uparrow\infty$, for any $R\in \R$. 
 Hence we see that 
 \be\Eq(palm.101)
\P\left(  \ov\EE^t_{\s_e}([-R,\infty))>1 \big|\supp\ov\EE^t_{\s_e} \cap dy\neq \emptyset\right)\downarrow 0,
\ee
uniformly in $y\geq 0$, as
 $ t $ and then $\s_e$ tend to infinity. Next,  
 \bea\Eq(palm.102)
&& \int_0^\infty \P\left(\supp\ov\EE^t_{\s_e} \cap dy\neq \emptyset\big|\max_{k\leq n(t)} x_k(t)>\sqrt 2\s_e t\right)\\\nonumber
&& \leq
  \int_0^\infty \P\left(\max_{k\leq n(t)} x_k(t)\geq \sqrt 2\s_e t+y\big|\max_{k\leq n(t)} x_k(t)>\sqrt 2\s_e t\right).
  \eea
But by Proposition 7.5 in \cite{BovHar13} the probability in the integrand converges to 
$\exp(-\sqrt 2\s_e y)$, as $t\uparrow \infty$.
 It follows from the proof that this convergence is unifomr in $y$, and hence by dominated convergence,  
the right-hand side of \eqv(palm.102) is finite. Therefore, \eqv(palm.1) holds. As  a consequence,  $\L_{\s_e}$ converges to $\d_0$. 
 
 It remains to show that the intensity of the Poisson process converges as claimed. 
  Theorems 1 and  2 of \cite{chauvin88} relate the constant $\wt C(\s_e)$ defined by
  \eqv(constant.1) to the intensity of the shifted BBM conditioned to exceed the level
  $\sqrt 2 \s_et$ as follows:
  \bea\Eq(cr.1)
 \frac 1{ \sqrt {4\pi} \wt C(\s_e)}&=&\lim_{s\uparrow\infty} \frac{\E\left[ \sum_{k} \1_{\bar 
 x_k(s) >\sqrt 2 \s_es}\right]}{\P\left( \max_k \bar x_k(s)> \sqrt 2\s_e s\right)
 }\\\nonumber
 &=&\lim_{s\uparrow\infty} 
 {\E\left[ \sum_{k} \1_{\bar x_k(s)-\max_i \bar x_i(s) >\sqrt 2 \s_es-\max_i\bar x_i(s)}\big|
 \max_k \bar x_k(s)> \sqrt 2\s_e s\right]}\\\nonumber
 &=& \L_{\s_e}((-E,0]),
 \eea
 where, by Theorem 7.5 in \cite{BovHar13}, $E$ is a exponentially distributed random 
 variable with parameter $\sqrt 2\s_e$, independent of $\L_{\s_e}$.  
 As we have just shown  that $\L_{\s_e} \rightarrow \d_0$, it follows that the right-hand 
 side tends to one, as $\s_e\uparrow \infty$, and hence
  $\wt C(\s_e)\uparrow 1/\sqrt {4\pi}$.
   Hence the intensity measure of the PPP appearing in $\EE_{\s_b,\s_e}$ converges to 
   the desired intensity measure $\frac{1}{\sqrt{4\pi}}Y_{\s_b}\eee^{-\sqrt{2}x}dx$.\\
\end{proof}
This proves Assertion (iii) of Theorem \thv(maintheo).
\end{proof}

\section{Proof of Proposition \thv(prop.number)}

We prove Proposition \thv(prop.number) via convergence of Laplace transforms.
For $u_1,\dots,u_k\in \R, k\in\N$, define the Laplace transform of $\{\NN_{u_1}(t),\dots,\NN_{u_k}(t)\}$,
\be\Eq(comp.9)
\LL_{u_1,\dots,u_k}(t,c)= \E\left(\exp\left(-\sum_{l=1}^k c_l \NN_{u_l}(t)\right)\right),\quad c=(c_1,\dots,c_k)^{t}\in \R_+^k,
\ee
and analogously  the Laplace transform 
$\LL_{u_1,\dots,u_k}(c)$ of $\{\NN_{u_1},\dots,\NN_{u_k}\}$. Proposition \thv(prop.number)  is then a consequence of the next proposition.

\begin{proposition}\TH(prop.Laplace) For any $k\in \N$, $u_1,\dots,u_k\in\R$ and $c_1,\dots,c_k\in\R_+$
\be\Eq(comp.10)
\lim_{t\to\infty}\LL_{u_1,\dots,u_k}(t,c)=\LL_{u_1,\dots,u_k}(c).
\ee
 \end{proposition}

The proof of Proposition \thv(prop.Laplace) comes in two main steps. First, we prove
the result for the case of two speed BBM. This was done in our previous paper \cite{BovHar13}. In fact, we will need a slight extension of that result where we 
allow a slight dependence of the speeds on $t$. This will be given in the next subsection.

The second step is to show that, under the hypotheses of Theorem \thv(maintheo),
  the Laplace transforms can be well approximated by those of two speed BBM. This 
  uses  the  classical  Gaussian comparison argument in a slightly subtle way.

\subsection{Approximating two speed BBM. The case $A'(1)<\infty$.}

It turns out that  it is enough to compare the process with covariance function $A$ 
with processes whose covariance function is piecewise linear with a  single change
 in slope. We  derive asymptotic  upper and lower bounds by choosing these
in such a way that the covariances near zero and near one are below, 
respectively above,
 that of the original process. 
We define 
\bea\Eq(delta)
\d^<(t)&=&\sup \{x\in[0,1]:A(x)\leq t^{-2/3}\}\nonumber\\
\d^>(t)&=&1-\inf \{x\in[0,1]:A(x)\geq 1-t^{-2/3}\} 
\eea

\begin{figure}[hbtp]
    \subfigure[Case 1) $\s_b=0$ but $\lim_{t\uparrow\infty}\d^<(t)=0$]{\includegraphics[width=0.49\textwidth]{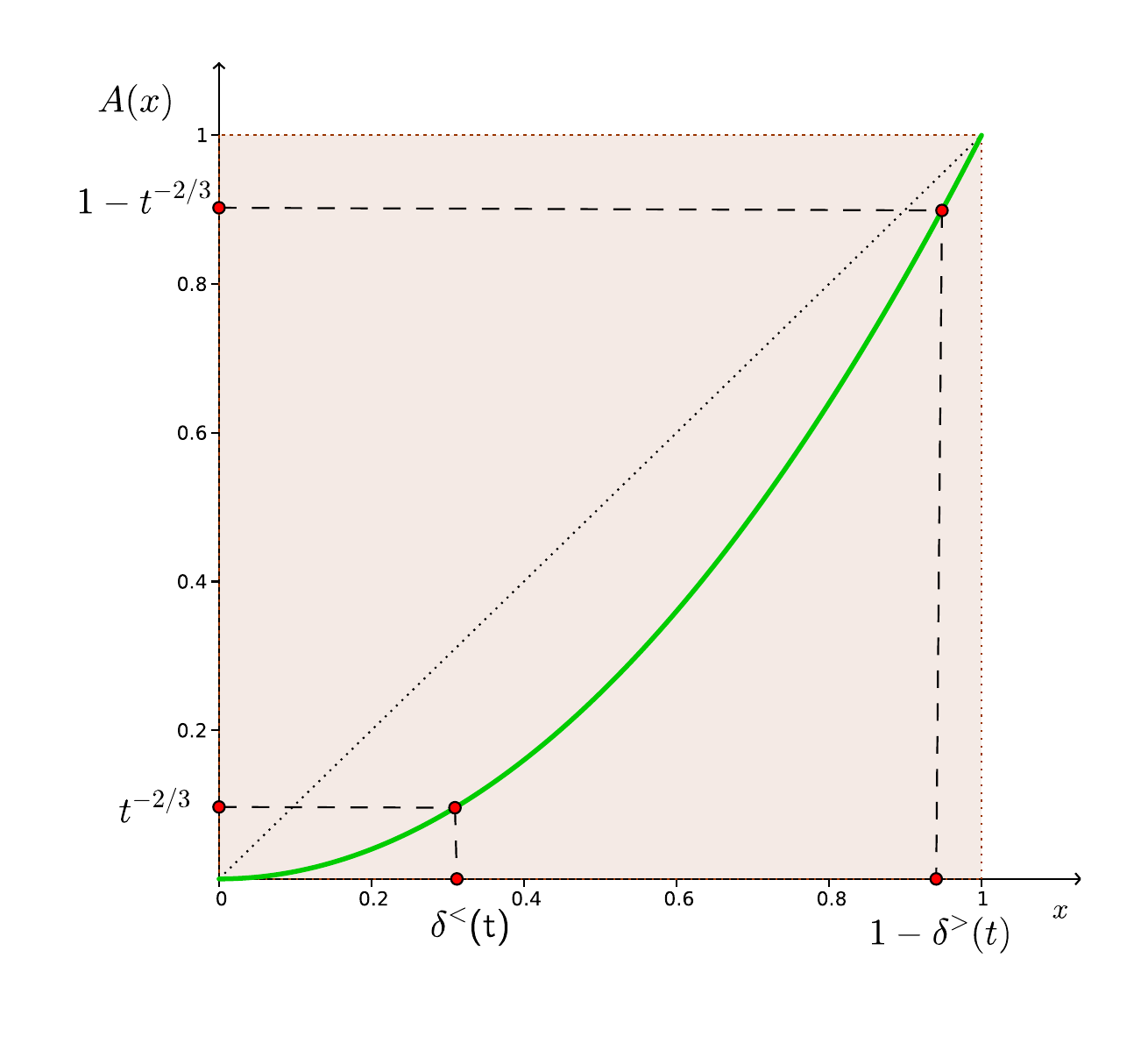} }
    \subfigure[Case 2) $\s_b=0$ but $\lim_{t\uparrow\infty}\d^<(t)=\d^<\neq0$] {\includegraphics[width=0.49\textwidth]{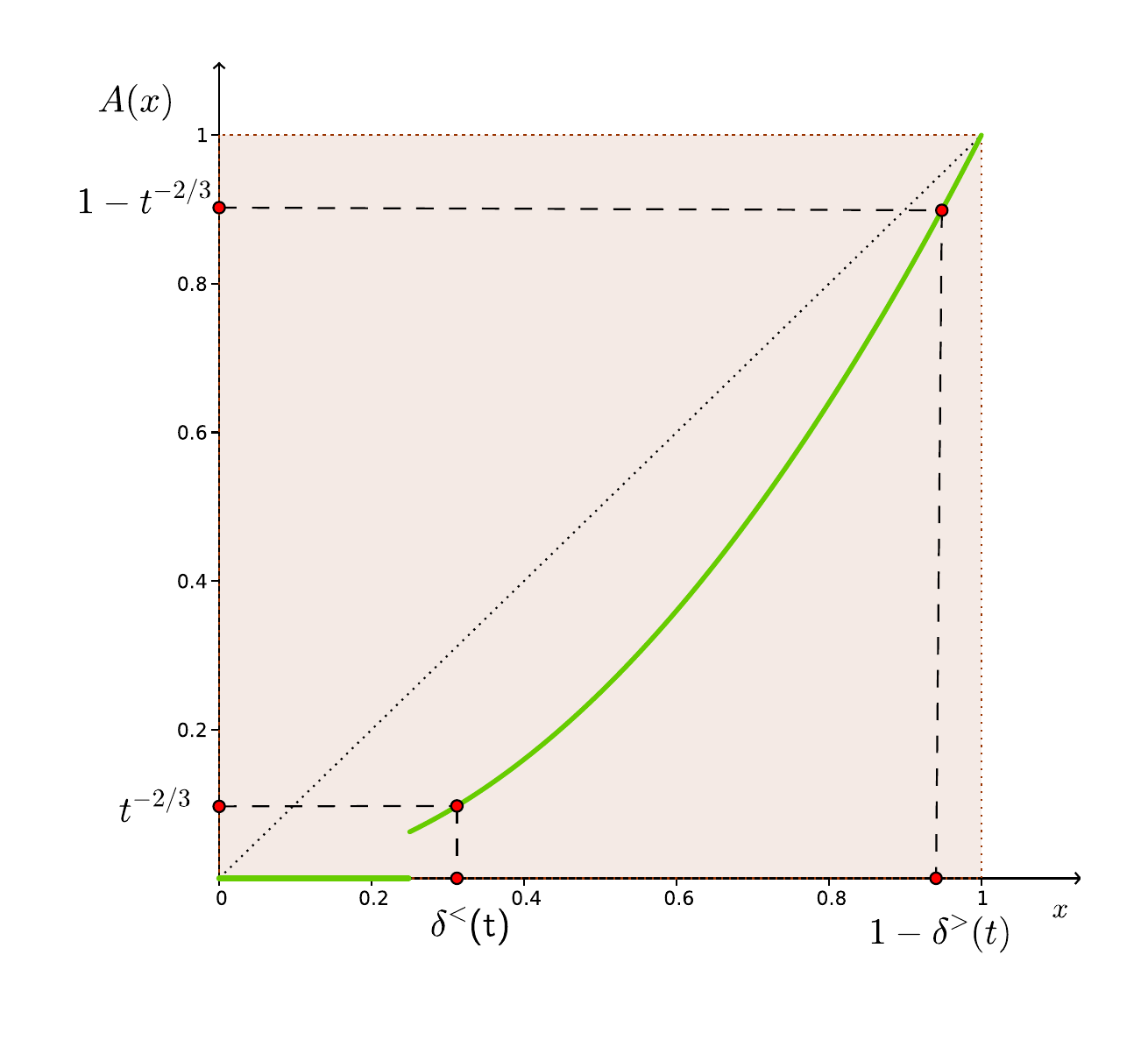}} 
\caption{Different cases for $\d^<(t)$ and $\d^>(t)$.} 
\label{fig.delta}
\end{figure} 
 
By Assumption (A1) it follows that $\lim_{t\uparrow\infty}\d^>(t)=0$.  

\begin{remark}
 If $\lim_{t\uparrow \infty}\d^<(t)=\d^<\neq 0$, then it follows from the
  definition of $\d^<(t)$ that $A(x)=0$ on $[0,\d^<]$.
\end{remark}

In the following formulas, we choose a parameter 
 $n\in\N_{\geq 2}$ as follows. If in Assumption (A2)  the functions  $\overline B, \underline B$ 
 can be chosen such that there exists  
  $m\geq 2$,  such that 
 $\underline{B}^{(k)}(0)=\overline{B}^{(k)}(0)=0$, for all $1\leq k<m$, and in some finite 
 interval $[0,\d_b]$, both $\vert\underline{B}^{(m)}(x)\vert$ and 
 $\vert\overline{B}^{(m)}(x)\vert$ are bounded by some constants $\underline{K}_1$, 
 respectively $\overline{K}_1$, then we choose $n$ as the largest of these integers. 
 Otherwise,  we choose $n=2$. Moreover, let
$\vert \overline{C}''(x)\vert \leq \overline{K}_2$ and $\vert \underline{C}''(x)\vert\leq \underline{K}_2$ for all $x\in[1-\d_e,1]$.
We define
\be\Eq(comp.1)
\overline{\Sigma}^2(s)=t\overline{A} (s/t)
\ee
and
\be\Eq(comp.2)
\underline{\Sigma}^2(s)=t\underline{A}(s/t).
\ee
Here 
\be\Eq(comp.3)
\overline{A}(x)=\begin{cases}
(\s_b^2+\frac{\overline{K}_1}{n!}(\d^<(t))^{n-1})x, & 0\leq x\leq \overline{b},\\
1+(\s_e^2-\frac{\overline{K}_2}{2}\d^>(t))(x-1), & \overline{b}< x\leq 1,               
                       \end{cases}
\ee
with 
\be\Eq(comp.4)
\overline{b}=\frac{1-\s_e^2+\frac{\overline{K}_2}{2}\d^>(t)}{\s_b^2+\frac{\overline{K}_1}{n!}(\d^<(t))^{n-1}-\s_e^2+\frac{\overline{K}_2}{2}\d^>(t)}.
\ee
If $\s^2_e<\infty$, 
\be\Eq(comp.4.1)
\underline{A}(x)=\begin{cases}
\left\{(\s_b^2-\frac{\underline{K}_1}{n!}(\d^<(t))^{n-1})x\right\} \vee 0, & 0\leq x\leq \underline{b},\\
1+(\s_e^2+\frac{\underline{K}_2}{2}\d^>(t))(x-1), & \underline{b}< x\leq 1,               
                       \end{cases}
\ee
with
\be\Eq(comp.6)
\underline{b}=\frac{1-\s_e^2-\frac{\underline{K}_2}{2}\d^>(t)}{\s_b^2-\frac{\underline{K}_1}{n!}(\d^<(t))^{n-1}-\s_e^2-\frac{\underline{K}_2}{2}\d^>(t)}.
\ee

\begin{remark}
If $\s^2_b=0$, $\underline A(x)=0$ for $0\leq x\leq \underline b$. 
If $\lim_{t\uparrow \infty}\d^<(t)=\d^<\neq 0$ (which implies that all derivatives in zero are $0$), we take  
\be\Eq(deriv.n3)
\overline{A}(x)=\begin{cases}
0, & 0\leq x\leq \overline{b},\\
1+(\s_e^2-\frac{\overline{K}_2}{2}\d^>(t))(x-1), & \overline{b}< x\leq 1,               
\end{cases}
\ee
and 
\be\Eq(deriv.n4)
\overline{b}=\frac{1-\s_e^2+\frac{\overline{K}_2}{2}\d^>(t)}{ -\s_e^2+\frac{\overline{K}_2}{2}\d^>(t)}.
\ee
If $A'(1)=\s^2_e=+\infty$,  then $\underline b=1$. And $\overline A\equiv \overline A_\rho$ is defined by 
\be\Eq(comp.3')
\overline{A}(x)=\begin{cases}
(\s_b^2+\frac{\overline{K}_1}{n!}(\d^<(t))^{n-1})x, & 0\leq x\leq \overline{b},\\
1+\rho (x-1), & \overline{b}< x\leq 1,               
                       \end{cases}
\ee
and $\overline{b}\equiv\overline{b}_\rho=\frac{1-\rho}{\s_b^2+\frac{\overline{K}_1}{n!}(\d^<(t))^{n-1}-\rho }$.
\end{remark}

The choice of $\overline{\Sigma}^2$ and $\underline{\Sigma}^2$ is motivated by the following properties.

\begin{lemma}\TH(lem.properties) 
$\overline{A}$ and $\underline{A}$ are piecewise linear, continuous functions with $\overline{A}(0)=\underline{A}(0)=0$ and $\overline{A}(1)=\underline{A}(1)=1$. Moreover,

\begin{itemize}
\item[(i)]
If $\lim_{t\uparrow \infty}\d^<(t)=0$, then,  
for all $s$ with $\Sigma^2(s)\in[0,t^{1/3}]$ and $\Sigma^2(s)\in[t-t^{1/3},t]$, 
\be\Eq(lem.pro1)
\overline{\Sigma}^2(s)\geq \Sigma^2(s)\geq \underline{\Sigma}^2(s).
\ee
\item[(ii)]
 If $\lim_{t\uparrow \infty}\d^<(t)=\d^<>0$, then \eqv(lem.pro1) only holds for all $s$ with $\Sigma^2(s)\in[t-t^{1/3},t]$ while, for $s\in[0,(\d\land \overline b)t)$, 
  \be\Eq(lem.pro1.1)
  \overline{\Sigma}^2(s)=\Sigma^2(s)=\underline{\Sigma}^2(s)=0.
  \ee 
\end{itemize} 
\end{lemma}

\begin{proof} $\overline{A}$ and $\underline{A}$ are obviously piecewise linear. 
The fact that they are continuous is easily verified. 
By definition, $A'(0)=\s_b^2$ and $A'(1)=\s_e^2$.
 For all $s$ such that  $\Sigma^2(s)\in[0,t^{1/3}]$, a $n$th-order Taylor expansion of 
 $\overline{B}$ at $0$ together with the fact that by assumption, for $k<n$, 
 $\overline{B}(0)=\overline{B}^{k}(0)=0$ shows that
 \be\Eq(lem.pro2)
 \Sigma^2(s)\leq \overline{B}(s)=t\left[\overline{B}'(0)\frac{s}{t}+\frac{\overline{B}^{(n)}(\xi)}{n!}\left(\frac{s}{t}\right)^n\right], \quad\mbox{for some }\xi\in(0,s).
 \ee
The reverse inequality holds when $\overline{B}$ is replaced by $\underline{B}$. Eq. \eqv(lem.pro1) follows then from  Assumption (A2).
 Using a second order Taylor expansion of $\overline{C}$ and $\underline{C}$ at $1$, we obtain Eq. \eqv(lem.pro1) for $\Sigma^2(s)\in[t-t^{1/3},t ]$.
Eq. \eqv(lem.pro1.1) holds trivially in the specified interval.
This concludes the proof of the lemma.
\end{proof}

Let $\{\overline{y}_i,i\leq n(t)\}$ be the particles of a BBM with speed function $
\overline{\Sigma}^2$ and let\\ $\{\underline{y}_i,i\leq n(t)\}$ be particles of a BBM with speed 
function $\underline{\Sigma}^2$. 
We want to show that the limiting extremal processes of these processes coincide.
Set
\bea\Eq(comp.7)
\overline{\NN}_{u}(t)\equiv\sum_{i=1}^{n(t)}\1_{\overline{y}_i(t)-\tilde m(t)>u},\\
\underline{\NN}_{u}(t)\equiv\sum_{i=1}^{n(t)}\1_{\underline{y}_i(t)-\tilde m(t)>u}.
\eea
 
\begin{lemma}\TH(lem.limit) 
 For all $u_1,\dots, u_k$ and all $c_1,\dots,c_k\in\R_+$,  the limits
 \be\Eq(lim.1)
 \lim_{t\uparrow \infty}\E\left(\exp\left(-\sum_{l=1}^k c_k \overline{\NN}_{u_l}(t)\right)\right)
  \ee
 and 
 \be\Eq(lim.2)
 \lim_{t\uparrow \infty}\E\left(\exp\left(-\sum_{l=1}^k c_k \underline{\NN}_{u_l}(t)\right)\right)
 \ee
 exist. If $A'(1)<\infty$, then   two limits coincide with $\LL_{u_1,\dots,u_k}(c)$. 
 
 If $A'(1)=\s_e^2=\infty$, then the two limits in \eqv(lim.1) and \eqv(lim.2)  converges to 
 the same limit, as 
  $\rho\uparrow \infty$.
\end{lemma}

\begin{proof}  We first consider the case when 
$A'(1)<\infty$. To prove  Lemma \thv(lem.limit), we show that the extremal processes
\be
\overline{\EE_t}=\sum_{i=1}^{n(t)}\d_{\overline{y}_i-\tilde m(t)}
\quad\mbox{and}\quad
\underline{\EE_t}=\sum_{i=1}^{n(t)}\d_{\underline{y}_i-\tilde m(t)}
\ee
both converge to $\EE_{\s_b,\s_e}$, that was defined in \eqv(maintheo.2). Note that this implies first convergence of Laplace functionals with 
 functions $\phi$ with compact support, while the   $\NN_u(t)$ have support that is unbounded from above. Convergence for these, however, carries over due to the tightness established in Proposition \thv(prop.tight).

To do so, 
 observe that the slopes at $0$ of $\overline{\Sigma}^2$ and $\underline{\Sigma}^2$ are equal to 
 $\s_b^2$ up to an error of order $ \d^<(t)$. Moreover,
  the slope at $t$ is in both cases, up to an error of order $ \d^>(t)$, equal to  $\s_e^2$. The time of 
  speed change $\overline{b}(t)$, respectively $\underline{b}(t)$, is equal to $\frac{1-\s_e^2}
  {\s_b^2-\s_b^2}$ up to an error of order $\d^>(t)\vee \d^<(t)$.
 For the two-speed BBM with speed
 \be\Eq(lim.10)
 \s^2(s)=\begin{cases}
 \s^2_b, &\,\hbox{for}\,\,0<s\leq \frac{1-\s_e^2}{\s_b^2-\s_e^2} ,\\
 \s^2_e, &\,\hbox{for}\,\,\frac{1-\s_e^2}{\s_b^2-\s_e^2}<s<t,
 \end{cases}
 \ee
 it was shown in \cite{BovHar13} that the maximal displacement is equal to $\tilde m(t)$ and that the extremal process converges to $\EE_{\s_b,\s_e}$ as $t\uparrow \infty$. 
 The method used to show this is to show the convergence of the Laplace functionals,
  $\E(\exp(-\int \phi(x) \EE_t(dx)))$, $\phi\in C_c(\R,\R_+)$.
  The other difference is that the function $A$ we have
 to consider now depend (weakly) on $t$. We need to show that this has no effect. 
 
 Inspecting the proof of the convergence of the Laplace functional, respectively convergence of 
 the maximum in \cite{BovHar13}, one sees that nothing changes (since  we keep $t$ fixed)
  until  Eq. (5.28) in  \cite{BovHar13}.
   There, we then have to show that, for each $y\in\R$, (in the case of $\overline{\Sigma}^2$)
 \be\Eq(lim.11)
 \E\left(\exp\left(-C(a)\left(\sfrac{\s_e^2-\frac{\overline{K}_2}{2}\d^>(t)}{1-[\s_b^2+
 \frac{\overline{K}_1}{2}\d^<(t)]\overline{b}/\sqrt{t}}\right)^{1/2}\eee^{-\sqrt{2}y}
 \overline{Y}_{\s_b,\overline{b}\sqrt t, \g}^B \right)(1+o(1))\right),
 \ee
 converges, as first $t\uparrow\infty$ and then $B\uparrow\infty$, to 
 \be\Eq(5.28.1)
 \E\left(\exp\left(-\s_e C(a)Y_{\s_b} \eee^{-\sqrt 2y}\right)\right),
 \ee
 where $C(a)>0$ is a constant depending on $a=\sqrt 2(\s_e-1)$ 
 (see \eqv(constant.1)), and
 \be\Eq(lim.12)
 \overline{Y}_{\s_b,\overline{b}\sqrt t, \g}^B=\sum_{i=1}^{n(\overline{b}\sqrt t)}\eee^{-(1+\s_b^2+
 \frac{K}{2}\d^<(t))\overline{b}\sqrt t+\sqrt{2 }\overline{y}_i(\overline b\sqrt t)}
 \1_{\overline{y}_i(\overline b\sqrt t)-\sqrt{2}(\s_b^2+\frac{K}{2}\d^<(t))\overline{b}\sqrt t\in[-Bt^{\g/2},Bt^{\g/2}]}.
 \ee
 The main task is to ensure the convergence of $\overline{Y}_{\s_b,\overline{b}\sqrt t, \g}^B $
 to the limit of the corresponding McKean martingale, $Y_{\s_b}$.
 In the case where $\lim_{t\uparrow\infty}\d^<(t) > 0$, this takes the simple form 
 \be\Eq(lim.12.bis)
 \overline{Y}_{0,\overline{b}\sqrt t, \g}^B=\sum_{i=1}^{n(\overline{b}\sqrt t)}\eee^{-\overline{b}\sqrt t},
 \ee
 which converges to an exponential random variable of mean one, as desired.

In the case when  $\lim_{t\uparrow\infty}\d^<(t) = 0$, a further slight modification is necessary.  
Observe that in the proof of Theorem 5.1 
in \cite{BovHar13}, $\overline b\sqrt t$ can be replaced by  any sequence $\D(t)\uparrow \infty$ 
such that $\lim_{t\uparrow\infty}\D(t)/t=0$. Here we adapt $\D(t)$ to the function $\Sigma^2$ and 
choose
 \be\Eq(important)
 \D(t)=\left(\d^<(t)\right)^{-1/2}.
 \ee
 Doing so, we have to show that,  analogously to \eqv(lim.11), the object 
 \be\Eq(lim.11.1)
 \E\left(\exp\left(-C(a)\left(\sfrac{\s_e^2-\frac{\overline{K}_2}{2}\d^>(t)}{1-[\s_b^2+
 \frac{\overline{K}_1}{2}\d^<(t)] \D(t)/t}\right)^{1/2}\eee^{-\sqrt{2}y}
 \overline{Y}_{\s_b,\D(t), \g}^B \right)(1+o(1))\right)
 \ee
 converges to \eqv(5.28.1).
 By our choice of $\D(t)$, 
  \be\Eq(lim.14)
 \left| \eee^{-\frac{\overline{K}_1}{2}\d^<(t)\D(t)}\eee^{\sqrt{2}\frac{\overline{K}_1}{2}\d^<(t)(\D(t)+B\D(t)^\g)}-1\right|\leq  const. \sqrt {\d^<(t)},
 \ee
which tends to zero, as $t\uparrow\infty$. Thus
\be\Eq(pseudo.1)
 \overline{Y}_{\s_b,\D(t), \g}^B =  \wt Y_{\s_b,\D(t),\g}^{B}(1+o(1)),
 \ee
 where
 \be\Eq(lim.13)
 \wt Y_{\s_b,\D(t),\g}^{B}\equiv 
 \sum_{i=1}^{n(\D(t))}\eee^{-(1+\s_b^2)\D(t)+\sqrt{2} 
 \s_b\overline{x}_i(\D(t))}\1_{\s_b\overline{x}_i(\D(t))-\sqrt{2}\s_b^2\D(t)\in[-B\D(t)^{\g},B
 \D(t)^{\g}]}.
 \ee
By Lemma 4.3 in \cite{BovHar13}, it follows that 
 $\wt Y_{\s_b,\D(t),\g}^{B}$ converges in probability and in $L^1$ to the random variable $Y_{\s_b}
 $. Since $\wt Y_{\s_b,\D(t),\g}^{B}\geq 0$ and $  C(a)>0$,
and since 
\be
\lim_{t\uparrow \infty}\left(\sfrac{\s_e^2-\frac{\overline{K}_2}{2}\d^>(t)}{1-[\s_b^2+\frac{\overline{K}_1}{2}\d^<(t)] \D(t)/t}\right)^{1/2} =\s_e,
\ee
 it follows that
  \bea\Eq(lim.16)
&& \lim_{B\uparrow \infty}\liminf_{t\uparrow\infty}\E\left(\exp\left(-  C(a)
\s_e
\eee^{-\sqrt{2}y}\wt{Y}_{\s_b,\D(t), \g}^B \right)(1+o(1))\right)\nonumber\\
&&=\lim_{B\uparrow \infty}\limsup_{t\uparrow\infty}\E\left(\exp\left(- C(a)
\s_e
\eee^{-\sqrt{2}y}\wt{Y}_{\s_b,\D(t), \g}^B \right)(1+o(1))\right)\nonumber\\
&&=\E\left(\exp\left(-\wt C(\s_e)Y_{\s_b}\eee^{-\sqrt 2 y}\right)\right),
 \eea
where $\wt C(\s_e)=\s_e C(a)$. The same arguments work when $\overline{\Sigma}^2$ is replaced by $\underline{\Sigma}^2$. The limit in \eqv(lim.16) coincides with the one obtained in \cite{BovHar13} for the two-speed BBM with speed given in \eqv(lim.10).
The assertion in the  case when $\s_e=\infty$ follows directly from Lemma \thv(lem.infinity).
 \end{proof}

 \subsection{Gaussian comparison}
We distinguish from now on  the expectation with respect to the underlying tree structure and the one with respect to the Brownian movement of the particles. 
\begin{itemize}
\item $ \E_n$: expectation w.r.t. Galton-Watson process.
\item $\E_B$: expectation w.r.t. the Gaussian process conditioned on the $\s$-algebra 
 $\FF_t^{\hbox{\tiny tree}}$ generated by the Galton Watson process.
\end{itemize}
For a given realisation of the Galton-Watson process, we now let $x$, $\bar y$,
and $\underline y$ be three 
independent Gaussian processes whose covariances are given as in \eqv(variance.2)
with $A$ replaced by $\overline A$ in the case of $\bar y$ and $\underline A$ in the 
case
of $\underline y$.

The proof of Proposition \thv(prop.Laplace) is based on the following Lemma that compares the Laplace transform $\LL_{u_1,\dots,u_k}(t,c)$
with the corresponding Laplace transform for the comparison processes.

\begin{lemma}\TH(lem.number)
 For any $k\in \N$, $u_1,\dots,u_k\in\R$ and $c_1,\dots,c_k\in\R_+$ we have 
  \bea\Eq(comp.8)
 \LL_{u_1,\dots,u_k}(t,c)&\leq& \E\left(\exp\left(-\sum_{l=1}^k c_l \overline{\NN}_{u_l}(t)\right)\right)+o(1)\\ \Eq(comp.11)
 \LL_{u_1,\dots,u_k}(t,c)&\geq&\E\left(\exp\left(-\sum_{l=1}^k c_l \underline{\NN}_{u_l}(t)\right)\right)+o(1)
 \eea
\end{lemma}

 \begin{proof}
 The proofs of \eqv(comp.8) and \eqv(comp.11) are very similar. Hence we focus on proving 
 \eqv(comp.8). We will, however, indicate what has to be changed when proving the lower bound 
 as we go along.
 For simplicity  all overlined names depend on $\overline{\Sigma}^2$. Corresponding quantities 
 where $\overline{\Sigma}^2$ is replaced by $\underline{\Sigma}^2$ are underlined.
 
 To use Gaussian comparison methods, we first replace the functions $\NN_u(t), \overline 
 \NN_u(t)$ by smooth approximants: 
 \be\Eq(noweak.1)
 \chi^\k(x)\equiv 
 \frac{1}{\sqrt {2\pi\k^2}}\int_{-
 \infty}^{x} \eee^{-z^2/2\k^2}dz,
 \ee
 \be\Eq(noweak.2)
 \NN^\k_u(t)\equiv \sum_{i=1}^{n(t)} \chi^\k(x_i(t)-\tilde m(t)-u),
 \ee
 and 
 \be\Eq(noweak.3)
\overline{\NN}^\k_u(t)\equiv \sum_{i=1}^{n(t)} \chi^\k(\bar y_i(t)-\tilde m(t)-u).
 \ee
 Note that, as $\k\downarrow 0$,
 \be\Eq(noweak.4)
 \chi^\k(x)\rightarrow \1_{x>0},\quad \NN_u^\k(t)\rightarrow \NN_u(t),\quad  
 \overline{\NN}_u^\k(t)\rightarrow \overline{\NN}_u(t).
 \ee
 In order to prove \eqv(comp.8), we show that for all $\k>0$,
 \be\Eq(noweak.0)
 \E_B\left(\exp\left(-\sum_{l=1}^k c_l {\NN}^\k_{u_l}(t)\right)\right)\leq 
 \E_B\left(\exp\left(-\sum_{l=1}^k c_l \overline{\NN}^\k_{u_l}(t)\right)\right)+R(t),
 \ee
 where $R(t)$ is independent of $\k$ and $\lim_{t\uparrow \infty} 
 \E R(t)=0$.
 
 From now on we work conditional on the $\s$-algebra generated by the 
 Galton-Watson tree. 
  We introduce the family of functions $f_{t,\k}:\R^{n(t)}\rightarrow \R$ by 
 \be\Eq(comp.14)
 f_{t,\k} (x)\equiv
 f_{t,\k}(x_1,\dots,x_{n(t)})
 \equiv \exp\left(-\sum_{i=1}^{n(t)} \sum_{l=1}^kc_l  \chi^\k(x_i-\tilde m(t)-u_l)\right).
 \ee
We want to control 
 \bea\Eq(comp.12)
 &&\E_B\left(\exp\left(-\sum_{l=1}^k c_l  \NN^\k_{u_l}(t)\right)\right)-\E_B\left(\exp\left(-\sum_{l=1}^k c_l \overline{\NN}^\k_{u_l}(t)\right)\right)\nonumber\\
 &&=\E_B\left(f_{t,\k}(x_1(t),\dots,x_{n(t)}(t))\right)-\E_B\left(f_{t,\k}(\overline{y}_1(t),\dots,\overline{y}_{n(t)}(t)))\right.
 \eea
  Define for $h\in[0,1]$ the interpolating process
 \be\Eq(comp.13)
 x_i^h=\sqrt{h}x_i+\sqrt{1-h}\overline{y}_i, \quad h\in [0,1].
 \ee
 The interpolating process $\{x_i^h,i\leq n(t)\}$ is a  Gaussian process with the same underlying  
  branching structure and speed function
  \be\Eq(comp.rem)
  \Sigma_h^2(s)=h\Sigma^2(s)+(1-h)\overline{\Sigma}^2(s).
  \ee
 Then, \eqv(comp.12) is equal to 
 \be\Eq(comp.14)
 \E_B\left(\int_0^1 \frac{d}{dh}f_{t,\k}(x^h(t))dh\right),
 \ee
 where 
 \be\Eq(comp.15)
 \frac{d}{dh}f_{t,\k}(x^h(t))=\frac{1}{2}\sum_{i=1}^{n(t)}\frac{\partial}{\partial x_i}f_{t,\k}(x_1^h(t),
 \dots,x_{n(t)}^h(t))\left[\frac{1}{\sqrt h}x_i(t)-\frac{1}{\sqrt{1-h}}\overline{y}_i(t)\right],
 \ee
 and 
 derivative:
 \be\Eq(comp.16)
 \frac{\partial}{\partial x_i}f_{t,\k}(x_1^h(t),\dots,x_{n(t)}^h(t))
 =-\left(\sum_{l=1}^k  
 \frac {c_l }{\sqrt {2\pi\k^2}}\eee^{ -\frac{(x_i^h(t)-\tilde m(t)-u_l)^2}{2\k^2}}\right)
 f_{t,\k}(x_1^h(t),\dots,x_{n(t)}^h(t)).
 \ee
  The key idea is to introduce a localisation  condition on the path of $x_i^h
  $ into \eqv(comp.15) at this stage. Note that it is not surprising
  at this point, since localising paths has been  a crucial tool in almost all computations 
  involving BBM, see already Bramson's paper \cite{B_M}. To do so, we insert into  the right-hand 
  side of \eqv(comp.15) a one in the form
 \be\Eq(comp.17)
 1=\1_{x_i^h\in\TT_{t,\bar I,\Sigma_h^2}^{\gamma}}
 +\1_{x_i^h\not\in\TT_{t,\bar I,\Sigma_h^2}^{\gamma}},
 \ee
 with 
 \be\Eq(i.1)
 \bar I \equiv \left[t(\d^<_0(t)\land \d^<_1(t)),t(1-\d^>_1(t))\right],
 \ee
 and $\TT_{t,I,\Sigma_h^2}^{\gamma}$  defined in \eqv(path.2). 
 Here $\d_1^{<,>}(t)\equiv \d^{<,>}(t)$, while $\d_0^{<,>}$ is defined 
 in the same way, but with respect to the speed function $\overline \S^2$ instead of 
 $\S^2$.  We call the two resulting summands  $\overline{S}^h_{<}$  and  
 $\overline{S}^h_{>}$, 
 respectively.

Note that, when proving the  lower bound, we choose instead of $\bar I$, the interval
\be\Eq(i.2)
 \underline I \equiv \left[t(\d^<_0(t)\land \d^<_1(t)),t(1-\d^>_0(t))\right].
 \ee
 
 The next lemma shows that $\overline{S}^h_{>}$ does not contribute to the expectation in \eqv(comp.15), as $t\to\infty$.
 
 \begin{lemma}\TH(lem.zero) With the notation above, we have
 \be\Eq(comp.18)
  \lim_{t\to\infty}\E_n\left(\int_0^1\E_B(\vert \overline{S}^h_{>}\vert)dh\right)=0.
  \ee
 \end{lemma}
 The proof of this lemma will be postponed. 
  
 We continue with the proof of Lemma \thv(lem.number). We are left with controlling, for fixed 
 $h\in (0,1)$,
 \be\Eq(comp.27)
 \E_B(\overline{S}^h_<)=\E_B\left(\frac{1}{2}\sum_{i=1}^{n(t)}\frac{\partial}{\partial x_i}f_{t,\k}(x^h(t))\1_{x_i^h \in \TT_{t,\bar I,\Sigma_h^2}^\g}\left[\frac{x_i(t)}{\sqrt h}-\frac{\overline{y}_i(t)}{\sqrt{1-h}}\right]\right).
 \ee
By the definition of $\TT_{t,\bar I,\Sigma_h^2}^{\gamma}$, 
 \be\Eq(comp.28)
 \1_{x_i^h \in \TT_{t,\bar I,\Sigma_h^2}^\g}= \1_{\forall s\in \bar I: \left\vert \xi_i^h(s)\right\vert\leq (\Sigma_h^2(s)\wedge  (t-\Sigma_h^2(s)))^\g},
 \ee
 where $\xi^h_i(s)\equiv x_i^h(s)-\frac {\S_h^2(s)}{t}x^h_i(t)$ is 
 a time changed Brownian bridge from $0$ to $0$ in   time $t$, which, as we 
 recall,  is independent of the endpoint  $x_i^h(t)$. 
 We want to apply a Gaussian integration by parts formula  to \eqv(comp.27).
 However, we need to take care of the fact that  each summand in \eqv(comp.27) depends on the 
 whole path of $\xi_i$ through the term in \eqv(comp.28). Therefore, we first approximate that 
 indicator function in \eqv(comp.28) by a discretised version. Let,  for $N\in\N$,   
 $t_1,\dots,t_{2^N}$ be  a sequence of equidistant points  in $[0,t]$. Define the following 
 sequence of approximations, $G_{h,N}: C(\R_+)\rightarrow \R$,  to the indicator function in \eqv(comp.28),
 \be\Eq(comp.200)
G_{h,N}(x)\equiv  g_h(x(t_1),\dots,x(t_{2^N})),
\ee
where
\bea\Eq(comp.201)
 g_{h}(z_1,\dots,z_{2^N})
&=&\prod_{\ell=1}^{2^N}\Bigl[\1_{t_\ell\in \bar I}
 \chi^{2^{-N}}\left((\S^2_h(t_\ell)\wedge 
 (t-\Sigma_h^2(t_\ell)))^\g-z_\ell \right)
 \\\nonumber&\times& \chi^{2^{-N}}\left((\S^2_h(t_\ell)\wedge 
 (t-\Sigma_h^2(t_\ell)))^\g+z_\ell \right)
 +\1_{t_\ell\not\in \bar I}\Bigr].
 \eea
 Clearly $G_{h,N}\rightarrow \1_{x \in \TT_{t,\bar I,\Sigma_h^2}^\g}$, pointwise, as $N\uparrow
 \infty$, while the derivatives $\frac{\partial}{\partial z_\ell }g_h(z_1,\dots, z_{2^N})$ are bounded.
  By the Gaussian integration by parts formula (see, e.g.,
 \cite[Appendix A.5]{Tala1}), we have, for any given $N\in \N$, 
 \bea\Eq(comp.30)
&& \E_B\left(x_i(t)\frac{\partial}{\partial x_i}f_{t,\k}(x^h(t))
G_{h,N}(\xi^h)\right) \\
 &=& \sum_{\ell=1}^{2^N} \E_B\left((x_i(t) \xi_i^h(t_\ell)\right)\E_B\left(f_{t,\k}(x^h(t))
 \frac{\partial}{\partial z_{ \ell}}g_{h}(\xi_i^h(t_1),\dots,\xi_i^h(t_{2^N}))
 \right)\nonumber\\\nonumber
 &&+  \sum_{j=1}^{n(t)}\E_B(x_i(t)x_j^h(t))\E_B\left(G_{h,N}(\xi^h)\frac{\partial^2}{\partial x_j\partial x_i}f_{t,\k}(x^h(t))
 \right).
 \eea
 But for all $\ell\in\{1,\dots,2^N\}$, 
 \bea\Eq(comp.30')
 \E_B\left(x_i(t)\xi_i^h(t_\ell)\right)&=& \sqrt h \E_B \left(x_i(t) x_i(t_\ell)-x_i(t) \frac {\S^2(t_\ell)} t 
 x_i(t)\right)\\\nonumber &=&\sqrt h\left(\S^2(t_\ell)-\S^2(t_\ell)\right)= 0,
 \eea
 and hence the second line in \eqv(comp.30) is equal to zero. 
 In exactly the same way we get
 \bea\Eq(comp.32.1)
 &&\E_B\left(\bar y_i\frac{\partial}{\partial x_i}f_{t,\k}(x_1^h(t),\dots,x_{n(t)}^h(t))\right)
\\\nonumber&& = \sum_{j=1}^{n(t)}\E_B(\bar y_i(t)x_j^h(t))\E_B\left(G_{h,N}(\xi^h)\frac{\partial^2}{\partial x_j\partial x_i}f_{t,\k}(x^h(t))\right).
 \eea
  Computing the covariances,  
 $ \E_B\left(x_i(t)x_j^h(t)\right)=\sqrt {h}\E \left(x_i(t)x_j(t)\right)$ and 
 $\E_B\left(\bar y_i(t)x_j^h(t)\right)=\sqrt {1-h}\E \left
 (\bar y_i(t)\bar y_j(t)\right)$, we obtain that 
 \bea\Eq(noweak.5)
 &&\E_B\left(\frac{1}{2}\sum_{i=1}^{n(t)}\frac{\partial}{\partial x_i}f_{t,\k}(x^h(t))
 G_{h,N}(\xi^h)\left[\frac{x_i(t)}{\sqrt h}-\frac{\overline{y}_i(t)}{\sqrt{1-h}}\right]\right)
\\\nonumber
&&=\sum_{\stackrel{i,j=1}{i\neq j}}^{n(t)}\left[\E_B(x_i(t)x_j(t))-\E_B(\overline{y}_i(t)\overline{y}_j(t))\right]\E_B\left(G_{h,N}(\xi^h)\frac{\partial^2f_{t,\k}(x^h(t))}{\partial x_i\partial x_j}\right),
 \eea
where crucially the terms with $i=j$ have cancelled. 
This equation holds for any $N$, 
and since $0\leq G_{h,N}(x)\leq 1 $, and the integral 
$\E_B\left(\frac{\partial^2f_{t,\k}(x^h(t))}{\partial x_i\partial x_j}\right)$ is finite 
(trivially, since
the mixed second derivatives of $f$ are bounded), by Lebesgue's dominated 
convergence theorem, the right hand side converges  to the expression where
$G_{h,N}$ is replaced by its limit. 
Similarly, in the left hand side we can apply Lebesgue's theorem,
majorising the integrands by $C |x_i(t)|$, etc, which are all integrable. 
Thus we obtain that 
\bea\Eq(noweak.7)
 &&\E_B\left(\frac{1}{2}\sum_{i=1}^{n(t)}\frac{\partial}{\partial x_i}f_{t,\k}(x^h(t))
 \1_{\1_{x_i^h \in \TT_{t,\bar I,\Sigma_h^2}^\g}}\left[\frac{x_i(t)}{\sqrt h}-\frac{\overline{y}_i(t)}{\sqrt{1-h}}\right]\right)
\\\nonumber
&&=\sum_{\stackrel{i,j=1}{i\neq j}}^{n(t)}\left[\E_B(x_i(t)x_j(t))-\E_B(\overline{y}_i(t)\overline{y}_j(t))\right]\E_B\left(\1_{x_i^h \in \TT_{t,\bar I,\Sigma_h^2}^\g}
\frac{\partial^2f_{t,\k}(x^h(t))}{\partial x_i\partial x_j}\right),
 \eea
 Introducing 
 \be\Eq(comp.35)
 1= \1_{d(x^h_i(t),x^h_j(t))\in\bar I}+\1_{d(x^h_i(t),x^h_j(t))\not\in\bar I},
 \ee
 into \eqv(noweak.7), we rewrite the right hand side of \eqv(noweak.7) as 
 $(\overline{T1})+(\overline{T2})$, where
 \bea\Eq(comp.36)
&&(\overline{T1})\\\nonumber
&&= \sum_{\stackrel{i,j=1}{i\neq j}}^{n(t)}\E_B\left[x_i(t)x_j(t)-\overline{y}_i(t)\overline{y}_j(t)\right]\E_B\left( \1_{d(x^h_i(t),x^h_j(t))\in\bar I}\1_{x_i^h\in\TT^\g_{t,\bar 
I,\Sigma^2_h}}\frac{\partial^2 f_{t,\k}(x^h(t))}{\partial x_i\partial x_j}\right), 
\eea
\bea
\Eq(comp.36.1)
&&(\overline{T2})\\\nonumber
&& =\sum_{\stackrel{i,j=1}{i\neq j}}^{n(t)}
\E_B\left[x_i(t)x_j(t)-\overline{y}_i(t)\overline{y}_j(t)\right]
\E_B\left( \1_{d(x^h_i(t),x^h_j(t))\not\in\bar I}
\1_{x_i^h\in\TT^\g_{t,\bar I,\Sigma^2_h}}\frac{\partial^2 f_{t,\k}(x^h(t))}{\partial x_i\partial x_j}\right).
 \eea
The term $(\overline{T1})$ is controlled by the following Lemma.
 
 \begin{lemma}\TH(lem.t1) With the notation above, there exists a constant 
 $\wt C<\infty$, independent of $t$ and $\k^2$, such that for all $t$ large and $\k^2$ small enough,
 \be\Eq(lem.t1.1)
  \left| \E_n\left(\int_0^1(\overline{T1})dh\right)\right|\leq \wt C \int_{\bar I}
  \left|\eee^{-s+\Sigma^2(s)+O(s^\g)}-\eee^{-s+\overline{\Sigma}^2(s)+O(s^\g)}\right|ds.
  \ee
 \end{lemma}
 
 Moreover, we have: 
 
 \begin{lemma}\TH(lem.t3) If  $\Sigma^2$ satisfies (A1)-(A3), and $\overline{\Sigma}^2$ is  as defined in \eqv(comp.1), then 
 \be\Eq(lem.t3.1)
  \lim_{t\to\infty}  \int_{\bar I}\left|\eee^{-s+\Sigma^2(s)+O(s^\g)}-\eee^{-s+\overline{\Sigma}^2(s)+O(s^\g)}\right|ds
  =0.
  \ee
 \end{lemma}
 We postpone the proofs of these lemmata to Section \ref{section.5}.
 
Up to this point the proof of \eqv(comp.11) works exactly as the proof of \eqv(comp.8) when $\overline{\Sigma}^2$ is replaced by $\underline{\Sigma}^2$.
 For $(\overline{T2})$ and $(\underline{T2})$ we have:
  
\begin{lemma}\TH(lem.t2) For almost all realisations  of the Galton-Watson process, the following statements hold:
\begin{itemize}
\item[(i)] If $\lim_{t\uparrow \infty}\d^<(t)=0$, then
 \be\Eq(lem.t2.1)
    (\overline{T2})\leq 0, 
  \ee
  and
   \be\Eq(lem.t2.12)
    (\underline{T2})\geq 0.
  \ee
  \item[(ii)] If $\lim_{t\uparrow \infty}\d^<(t)=\d^<> 0$, then 
  \be\Eq(lem.t2.123)
  \lim_{t\uparrow \infty}(\overline{T2})\leq 0,
  \ee
  and
  \be\Eq(lem.t2.1234)
  \lim_{t\uparrow \infty}(\underline{T2})\geq 0.
  \ee
  \end{itemize}
 \end{lemma}
 The proof of this lemma is again postponed to Section \ref{section.5}.
 
From Lemma \thv(lem.t1), Lemma \thv(lem.t3), and Lemma \thv(lem.t2) together with 
\eqv(comp.27), the bound  \eqv(noweak.0) follows. Since the left and right hand 
sides involve expectations over bounded functions, we may pass to the limit 
$\k^2\downarrow 0$. 
This implies \eqv(comp.8). As pointed out, using Lemma \thv(lem.t2), the bound \eqv(comp.11) also follows. Thus, Lemma \thv(lem.number) is proved, once we provide the postponed proofs of the various lemmata above.
 \end{proof}

We conclude the proof of Proposition \thv(prop.Laplace).

\begin{proof}[Proof of Proposition \thv(prop.Laplace)]
Taking the limit as $t\uparrow \infty$ in \eqv(comp.8) and \eqv(comp.11) 
and using Lemma \thv(lem.limit) gives, in the case $A'(1)<\infty$, 
\bea\Eq(lap.1)
\limsup_{t\uparrow \infty}\LL_{u_1,\dots,u_k}(t,c)&\leq& \LL_{u_1,\dots,u_k}(c)\\
\liminf_{t\uparrow \infty}\LL_{u_1,\dots,u_k}(t,c)&\geq& \LL_{u_1,\dots,u_k}(c)
\eea
 Hence $\lim_{t\uparrow  \infty}\LL_{u_1,\dots,u_k}(t,c)$ exists and is equal to $\LL_{u_1,\dots,u_k}(c)$.
 In the case $A'(1)=\infty$, the same result follows if in addition we take 
 $\rho\uparrow \infty$ after taking $t\uparrow \infty$.
  This  concludes the proof of Proposition \thv(prop.Laplace).
\end{proof}

\section{Proofs of the auxiliary lemmata} \label{section.5}
We now provide the proofs of the lemmata from the last section whose proofs we had postponed.

\begin{proof}[Proof of Lemma \thv(lem.zero)] We have
\be\Eq(comp.19)
\E_B(\vert \overline{S}^h_{>}\vert)\leq \frac12\sum_{i=1}^{n(t)}\sum_{l=1}^k c_l\E_B\left(
\frac {\eee^{-\frac{(x_i^h(t)-\tilde m(t)-u_l)^2}{2\k^2}}}{\sqrt {2\pi\k^2}}
\1_{x_i^h\not\in\TT_{t,\bar I,\Sigma_h^2}^{\gamma}}
\left[\sfrac{|x_i(t)|}{\sqrt{h}}+\sfrac{|\overline{y}_i(t)|}{\sqrt{1-h}}\right]\right).
\ee
We use the fact that the condition in the indicator function involves only the time 
changed
Brownian bridge, $\xi_i^h(s)=x_i^h(s)-\frac {\S^2_h(s)}t x^h_i(t)$, which is
 independent of the endpoint $x_i^h(t)$, and of course also of $x_i(t)$. This implies that
 \bea\Eq(noweak.10)
 &&\E_B\left(
\frac {\eee^{-\frac{(x_i^h(t)-\tilde m(t)-u_l)^2}{2\k^2}}}{\sqrt {2\pi\k^2}}
\1_{x_i^h\not\in\TT_{t,\bar I,\Sigma_h^2}^{\gamma}}
\sfrac{|x_i(t)|}{\sqrt{h}}\right)
\\\nonumber
 &&=\E_B\left(
\frac {\eee^{-\frac{(x_i^h(t)-\tilde m(t)-u_l)^2}{2\k^2}}}{\sqrt {2\pi\k^2}}\sfrac{|x_i(t)|}{\sqrt{h}}\right)
\P_B\left(x_i^h\not\in\TT_{t,\bar I,\Sigma_h^2}^{\gamma}\right),
\eea
and similarly for the terms involving $\bar y$. The computation of the first expectation is 
a straightforward Gaussian integration involving two independent Gaussians. In fact we can write 
\be\Eq(noweak.12)
\E_B\left(
\frac {\eee^{-\frac{(x_i^h(t)-\tilde m(t)-u_l)^2}{2\k^2}}}{\sqrt {2\pi\k^2}}|x_i(t)|\right)
=\int \frac{dz_1dz_2}{(2\pi)^{3/2} t\k}\eee^{-\frac 12 (\underline z, M\underline{z})+ (\underline {v},\underline {z}) - (\tilde m(t)+u_l)^2/2\k^2} |z_1|,
\ee
where 
\be\Eq(noweak.13)
M\equiv \left(\begin{matrix} \frac{\k^2+th}{t\k^2}&\sqrt{h(1-h)}/\k^2\\
\sqrt{h(1-h)}/\k^2 & \frac{\k^2+t(1-h)}{t\k^2}\end{matrix}\right),\qquad
\underline{v}\equiv \frac {\tilde m(t)+u_l}{\k^2}\left(\begin{matrix} \sqrt h\\\sqrt{1-h}
\end{matrix}\right).
\ee
Note that $\det M =t^{-2}+t^{-1}\k^{-2}$, and its eigenvalues are
given by
\be\Eq(noweak.14)
\l_{\pm} = t^{-1}+\k^{-2} \pm \sqrt{\k^{-4}+t^{-1}\k^{-2}}.
\ee
Importantly, the smaller eigenvalue behaves, when $\k^2/t$ tends to zero,
as
\be\Eq(noweak.14.1)
\l_-=1/(2t)\left(1+O(\k^2/t)\right).
\ee 
The remaining calculations amount to completing the square. With 
\be\Eq(noweak.15)
\underline {a}\equiv \frac {\tilde m(t)+u_l}{\k^2t^{-1}+1}\left(\begin{matrix} \sqrt h\\\sqrt{1-h}
\end{matrix}\right),
\ee
we can rewrite the right hand side of \eqv(noweak.12) as 
\be\Eq(noweak.16)
\frac{\eee^{- \frac12\frac{(\tilde m(t)+u_l)^2}{t+\k^2}}}{(2\pi)^{1/2}\sqrt{ \k^2+t}}\int \frac{dz_1dz_2}{2\pi \sqrt{\det({M^{-1}})}}
\eee^{-
\frac 12 \left(\underline {z}-\underline {a}, M(\underline{z}-\underline {a})\right)^2}
  |z_1|.
 \ee 
 Now it is plain that the last expectation is bounded by
\be\Eq(noweak.17)
| \underline {a}_1| + const.  (\l_{-})^{-1/2} \leq \sqrt h\frac{\tilde m(t)+u_l}{\k^2/t+1}
+2 t^{1/2}(1+ O(\k^2t^{-2})) \leq const. (\sqrt h t +\sqrt t),
\ee 
with the constant uniform in, say, $\k^2\leq 1, t\geq 100$. 
This allows us to bound \eqv(noweak.16) by a uniform constant times
\be\Eq(noweak.18)
\left(\sqrt {ht}+ 2\right) \eee^{-t/(1+\k^2/t) + \ln t/(1+\k^2/t)} \leq const.
\eee^{-t}\left(\sqrt {ht^3}+ 2t\right).
\ee
Next we bound the probability that the Brownian bridge does not stay in the tube.
For this we use Lemma \thv(BB.basic).
 Note that by construction, if $s\in \bar I$, then
for all $h\in [0,1]$, 
$\S^2_h\geq  Dt^{1/3}$, and $\S^2_h\leq t-Dt^{1/3}$,
for some constant $0<D<\infty$, depending only on the function $A$. Thus, 
by Eq. \eqv(path.3.1)) of  Lemma  \thv(BB.basic), 
\be\Eq(comp.26)
\P_B\left(x_i^h\not\in\TT_{t,\bar I,\Sigma_h^2}^{\gamma}\right)\leq 
8\sum_{k= Dt^{1/3}}^\infty k^{1/2-\g}\eee^{-k^{2\g-1}/2}.
\ee
We are now ready to insert everything back into \eqv(comp.19). 
This gives that, uniformly in $\k^2$ small and $t$ large (as above)
\be\Eq(noweak.17)
\E_B\left(|\overline{S}^h_>|\right)
\leq n(t)  const. \sum_{l=1}^k c_l \eee^{-t} \left(2\sqrt {t^3}+ 2t/\sqrt h +2t/\sqrt{1-h}\right)
\eee^{-D^{2\g-1}t^{(2\g-1)/3}}.
\ee
Integrating over $h$ and taking the expectation with respect to the Galton-Watson 
process yields
\be\Eq(noweak.18)
\E_n\left(\int_0^1\E_B\left(|\overline{S}^h_>|\right)dh\right)
\leq   const. \sum_{l=1}^k c_l  \sqrt {t^{3/2}}
\eee^{-D^{2\g-1}t^{(2\g-1)/3}},
\ee
which tends to zero as $t\uparrow \infty$, uniformly in $\k\leq 1$, as claimed, if $\g>1/2$.
 This proves the assertion of the lemma.
\end{proof}

 \begin{proof}[Proof of Lemma \thv(lem.t2)] We first proof \eqv(lem.t2.1).
 Observe that 
 \be 
 d(x_i(t),x_j(t))=d(\overline{y}_i(t),\overline{y}_j(t))=d(x_i^h(t),x_j^h(t)).
 \ee
 Moreover,  for all $1\leq i,j\leq n(t), i \neq j$, 
 \be\Eq(t2.1)
 \1_{x_i^h\in\TT^\g_{t,\bar I,\Sigma^2_h}}\frac{\partial}{\partial x_i\partial x_j}f_{t,\k}(x_1^h(t),\dots,x^h_{n(t)}(t)) \geq 0.
 \ee
 For $d(x_i(t),x_j(t))\in\left[0,t(\d_1^<(t)\land\d_0^<(t))\right)$, we distinguish the cases 
 $\lim_{t\to\infty} \d^<(t)> 0 $ and $\lim_{t\to\infty} \d^<(t)=0 $, respectively. 

 If $\lim_{t\to\infty} \d^<(t)=\d^<> 0 $, then
  $A(x)=\overline A(x)=\underline A(x)=0$,
   for all $x\in[0,t(\d_1^<(t)\land\d_0^<(t)))$. 
  Thus all the terms in both $(\overline{T2})$ and $(\underline{T2})$ with $i,j$ such that
  $ d(x_i(t),x_j(t))\in[0,t(\d_1^<(t)\land\d_0^<(t))) $ vanish.
   
 Next consider the case where $\lim_{t\to\infty} \d^<(t)=  0 $.
  By Lemma \thv(lem.properties) we have, for $\overline{y}_i(t), \overline{y}_j(t)$ with 
  $d(\overline y_i(t),\overline{y}_j(t))\in[0,t(\d_1^<(t)\land\d_0^<(t)))$, that
    \bea\Eq(t2.2)
  \E_B(\overline{y}_i(t),\overline{y}_j(t))
  &=& \overline{\S}^2\left(
  d(\overline{y}_i(t),\overline{y}_j(t))\right)\nonumber\\
   &\geq& \Sigma^2(d(\overline{y}_i(t),\overline{y}_j(t)))\nonumber\\
 &=&\Sigma^2(d(x_i(t),x_j(t)))=\E_B(x_i(t),x_j(t)).
  \eea
   For \eqv(lem.t2.12) we proceed in the same way but instead  of \eqv(t2.2) we have, for
   $d(\underline{y}_i(t),\underline{y}_j(t))\in[0,t(\d_1^<(t)\land\d_0^<(t)))$,
  \bea\Eq(t2.4)
  \E_B(\underline{y}_i(t),\underline{y}_j(t))
  &=& \underline{\S}^2\left(
  d(\underline{y}_i(t),\underline{y}_j(t))\right)\nonumber\\\nonumber
  &\leq& \Sigma^2(d(\underline{y}_i(t),\underline{y}_j(t)))\\
  &=&\Sigma^2(d(x_i(t),x_j(t)))=\E_B(x_i(t),x_j(t)).
  \eea

  If  $d(\overline{y}_i(t),\overline{y}_j(t))\in[t(1- \d_1^>(t)),t]$, resp. 
  $d(\underline{y}_i(t),\underline{y}_j(t))\in[t(1- \d_0^>(t)),t]$, 
  we obtain in both cases from  Lemma \thv(lem.properties) that
  \be\Eq(t2.3)
  \E_B(\overline{y}_i(t),\overline{y}_j(t))\geq \E_B(x_i(t),x_j(t)),
  \ee
  and 
  \be\Eq(t2.5)
  \E_B(\underline{y}_i(t),\underline{y}_j(t))\leq \E_B(x_i(t),x_j(t)), 
  \ee
  respectively. 
  This concludes the proof of Lemma \thv(lem.t2).
 \end{proof}

 \begin{proof}[Proof of Lemma \thv(lem.t1) ]
 We have that
  \bea\Eq(t1.1) 
  &&\left |\E_n\left(\int_0^1(\overline{T1})dh\right)\right|\leq
 \E_n\bigg(\sum_{\stackrel{i,j=1}{i\neq j}}^{n(t)}\left|\E_B(x_i(t)x_j(t))-
\E_B(\overline{y}_i(t)\overline{y}_j(t))\right|\\\nonumber
&&\qquad\times\int_0^1\E_B\left(\1_{d(x^h_i(t),x^h_j(t))\in\bar I}\1_{x_i^h\in\TT^\g_{t,\bar I,\Sigma^2_h}} \frac{\partial^2 f_{t,\k}(x^h(t))}{\partial x_i\partial x_j}\right)dh\bigg).
  \eea
   By definition of $f_{t,\k}$ we have for $i\neq j$
  \bea\Eq(t1.no.50)
  \frac{\partial^2 f_{t,\k}(x^h(t))}{\partial x_i\partial x_j} &=&\sum_{l,\bar l=1}^k 
  \frac{c_l 
  c_{\bar l}}{2\pi\k^2}\eee^{\frac{-(x_i^h(t)-\tilde m(t)-u_l)^2}{2\k^2}}
  \eee^{\frac{-(x_j^h(t)-\tilde m(t)-
  u_{\bar l})^2}{2\k^2}}f_{t,\k}(x^h(t))\nonumber\\
  &\leq& \sum_{l,\bar l=1}^k 
  \frac{c_l 
  c_{\bar l}}{2\pi\k^2}\eee^{\frac{-(x_i^h(t)-\tilde m(t)-u_l)^2}{2\k^2}}
  \eee^{\frac{-(x_j^h(t)-\tilde m(t)-
  u_{\bar l})^2}{2\k^2}},
  \eea
  where we used that $f_{t,\k}\leq 1$. Using this bound we get that \eqv(t1.1) is bounded from above by
  \bea\Eq(t1.no.51)
  &&\E_n\Biggl(\sum_{\stackrel{i,j=1}{i\neq j}}^{n(t)}\left|\E_B(x_i(t)x_j(t))-
\E_B(\overline{y}_i(t)\overline{y}_j(t))\right|
 \int_0^1\E_B\Biggl(\1_{d(x^h_i(t),x^h_j(t))\in\bar I}\1_{x_i^h\in\TT^\g_{t,\bar I,\Sigma^2_h}}\nonumber\\
&&\qquad\times \sum_{l,\bar l=1}^k 
  \frac{c_l 
  c_{\bar l}}{2\pi\k^2}\eee^{\frac{-(x_i^h(t)-\tilde m(t)-u_l)^2}{2\k^2}}
  \eee^{\frac{-(x_j^h(t)-\tilde m(t)-
  u_{\bar l})^2}{2\k^2}}\Biggr)dh\Biggr).
  \eea
  We introduce the shorthand notation
 \bea\Eq(t1.not)
 A_1&=&\Sigma_h^2(s)/t,\nonumber\\
 A_2&=&1-\Sigma_h^2(s)/t.
 \eea 
 To compute the expectation in \eqv(t1.no.51) we fix the time of the most recent common 
 ancestor of $x_i$ and $x_j$ as $s$ and integrate over it.  
 Then the pair  $(x_i^h(t),x_j^h(t))$ has the same distribution 
 as $(Y+X_1,Y+X_2)$, where $Y,X_1,X_2$ are independent centred 
 Gaussian random variables with variance $tA_1,tA_2$, and $tA_2$, respectively.
 We  also  relax 
 the tube condition except at the splitting time of the two particles. From this we see that 
the expression in  \eqv(t1.no.51) is bounded from above by
 \bea\Eq(t1.2)
 &&C \sum_{l,\bar l=1}^k c_l c_{\bar l}\int_{\bar I}|\Sigma^2(s)-\overline{\Sigma}^2(s)|\eee^{2t-s}\\\nonumber&&\quad\times\int_0^1\int_{A_1\tilde m(t)-J(s,\g)}^{A_1\tilde m(t)+J(s,\g)}
 Q(y,u_l,t) Q(y,u_{\bar l},t)
 \eee^{-\frac{y^2}{2tA_1}}\sfrac{dy}{\sqrt{2\pi t A_1}}dh ds,
 \eea
 where $\infty>C>0$ is a constant,  
 \be\Eq(t1.3)
 J(s,\g)= \left(\Sigma_h^2(s)\wedge (t-\Sigma_h^2(s))\right)^\g=( (A_1\wedge A_2)t)^\g,
 \ee
 and for $1\leq l\leq k$
 \be\Eq(t1.3.1)
 Q(y,u_l,t)=\int_{-\infty}^{\infty} \eee^{-(x+y-\tilde m(t)-u_l)^2/2\k^2}\eee^{-\frac{x^2}{2tA_2}}\sfrac{dx}{\sqrt{(2\pi)^2 \k^2 t A_2}}.
 \ee
 We first compute $ Q(y,u_l,t)$. We change variables in \eqv(t1.3.1)
 \be\Eq(t1.4)
 x=z+\frac{tA_2(\tilde m(t)-y-u_l)}{\k^2+tA_2}
 \ee
 and obtain that \eqv(t1.3.1) can be written as 
 \bea\Eq(t1.no.1)
 Q(y,u_l,t)
 &=&\eee^{-\frac{(\tilde m(t)-y-u_l)^2}{2(\k^2+tA_2)}}\int_{-\infty}^{\infty}\frac{\eee^{-\frac{z^2(\k^2+A_2t)}{2\k^2 A_2t}}}{ \sqrt{(2\pi)^2\k^2A_2 t}}dz \nonumber
 \\ &=& \frac{\eee^{-\frac{(\tilde m(t)-y-u_l)^2}{2(\k^2+tA_2)}}}{\sqrt{2\pi( \k^2+tA_2)}}.
  \eea
  Plugging this into \eqv(t1.2) we get 
  \bea\Eq(t1.no.3)
   &&C \sum_{l,\bar l=1}^k c_l c_{\bar l}\int_{\bar I}|\Sigma^2(s)-\overline{\Sigma}^2(s)
   |\eee^{2t-s}\\\nonumber&&\quad\times\int_0^1\int_{A_1\tilde m(t)-J(s,\g)}^{A_1\tilde m(t)+J(s,\g)}
  \frac{\eee^{-\frac{(\tilde m(t)-y-u_l)^2+(\tilde m(t)-y-u_{\bar l})^2}{2(\k^2+tA_2)}}
}{ 2\pi  (\k^2+tA_2)} \eee^{-\frac{y^2}{2tA_1}}\sfrac{dy}{\sqrt{2\pi t A_1}}dh ds,
  \eea
  In the integral with respect to $y$ we now  change variables to
  \be\Eq(t1.no.4)
  -w=y-\frac{(2\tilde m(t)-u_l-u_{\bar l})A_1t}{\k^2+(1+A_1)t},
  \ee
 and drop terms that are bounded  uniformly in $t$ and $\k^2$ by constants to see that
 \eqv(t1.no.3) is less than or equal to 
  \bea\Eq(t1.no.4)
  & & \wt C \sum_{l,\bar l=1}^k c_l c_{\bar l}\int_{\bar I}|\Sigma^2(s)-
  \overline{\Sigma}^2(s)|\eee^{2t-s}\int_0^1 \int_{\frac{A_1A_2t\tilde m(t)-A_1\tilde m(t)
  \k^2-tA_1(u_l+u_{\bar l})}{\k^2+(1+A_1)t}-J(s,\g)}^{\frac{A_1A_2t\tilde m(t)-A_1\tilde 
  m(t)\k^2-tA_1(u_l+u_{\bar l})}{\k^2+(1+A_1)t}+ J(s,\g)}\nonumber\\
 & &\quad\quad \times \frac{\eee^{-\frac{\tilde m(t)^2}{\k^2+(1+A_1)t}} \eee^{\frac{-
 w^2(\k^2+(1+A_1)t)}{2(\k^2+tA_2)A_1t}}}{ 2\pi  (\k^2+tA_2)}\sfrac{dwdhds}
 {\sqrt{2\pi t A_1}},
 \eea
with $\wt C$ a new constant independent of $t$ and $\k^2$.  
 Since,  for each $h\in(0,1)$,
 \bea\Eq(t1.100)
  &&\frac{\sqrt{1+A_1}}{\sqrt{tA_1A_2}}\left( \frac{A_1A_2\tilde m(t)}{A_1+1}-J(s,\g)\right)\nonumber\\
&\geq&((A_1\wedge A_2)\tilde m(t))^{-1/2}\left(\frac{1}{4}(A_1\wedge A_2)\tilde m(t)-(A_1\wedge A_2)^\g t^\g\right),
 \eea
which tends to $+\infty$, as $t\uparrow\infty$, 
we can use the Gaussian tail bound  \eqv(gaussian) 
in the  integral over $w$ to show that 
 \bea\Eq(t1.10)\nonumber
&&\eee^{2t-s} \int_0^1 \int_{\frac{A_1A_2t\tilde m(t)-A_1\tilde m(t)
  \k^2-tA_1(u_l+u_{\bar l})}{\k^2+(1+A_1)t}-J(s,\g)}^{\frac{A_1A_2t\tilde m(t)-A_1\tilde 
  m(t)\k^2-tA_1(u_l+u_{\bar l})}{\k^2+(1+A_1)t}+ J(s,\g)}
\frac{\eee^{-\frac{\tilde m(t)^2}{\k^2+(1+A_1)t}} \eee^{\frac{-
 w^2(\k^2+(1+A_1)t)}{2(\k^2+tA_2)A_1t}}}{ 2\pi  (\k^2+tA_2){\sqrt{2\pi t A_1}}}{dwdh}
 \\\nonumber
&&\leq\eee^{2t-s}\int_0^1 
 \eee^{-\frac{1}{2}\frac{\k^2+(1+A_1)t}{(\k^2+tA_2)A_1t}\left( \frac{A_1A_2t\tilde m(t)-A_1\tilde m(t)\k^2-tA_1(u_l+u_{\bar l})}{\k^2+(1+A_1)t}-J(s,\g) \right)^2}\frac{\eee^{-\frac{\tilde m(t)^2}{\k^2+(1+A_1)t}}}{ 2\pi  (\k^2+tA_2)}\nonumber\\
  & &\qquad\times   
  \left[\left(\sfrac{A_1A_2t\tilde m(t)-A_1\tilde m(t)
  \k^2-tA_1(u_l+u_{\bar l})}{\k^2+(1+A_1)t}-J(s,\g)\right)\sfrac{\k^2+(1+A_1)t}{(\k^2+tA_2)A_1t}\right]^{-1}
  \sfrac{dh}{\sqrt{2\pi t A_1}}.\nonumber\\
  &&=\eee^{2t-s}\int_0^1 
 \eee^{-\frac{1}{2}\frac{\k^2+(1+A_1)t}{(\k^2+tA_2)A_1t}\left( \frac{A_1A_2t\tilde m(t)-A_1\tilde m(t)\k^2-tA_1(u_l+u_{\bar l})}{\k^2+(1+A_1)t}-J(s,\g) \right)^2}\eee^{-\frac{\tilde m(t)^2}{\k^2+(1+A_1)t}} \nonumber\\
  & &\qquad\times   
  \sfrac{\sqrt{A_1 t}}{A_1A_2t\tilde m(t)-A_1\tilde m(t)
  \k^2-tA_1(u_l+u_{\bar l})-J(s,\g)(\k^2+(1+A_1)t)} \sfrac{dh}{ (2\pi)^{\frac{3}{2}}}.
 \eea
 By the definition of $J(s,\g)$ we can bound \eqv(t1.10) from above by 
 \be\Eq(t1.no.40)
  \wh C\eee^{2t-s}\int_0^1 \sfrac{\sqrt{A_1 t}}{A_1A_2t\tilde m(t)-A_1\tilde m(t)
  \k^2-tA_1(u_l+u_{\bar l})-J(s,\g)(\k^2+(1+A_1)t)}
  \eee^{-\frac{(1+A_2)\tilde m(t)^2}{2t}+O(s^\g)}\sfrac{dh}{ (2\pi)^{\frac{3}{2}}},
 \ee
 where $\wh C<\infty$ is a constant that does not depend on $t$ and $\k^2$. 
The denominator in the fraction appearing in the integrand 
   equals   $\sqrt 2A_2A_1t^2(1+o(1)$, for $t$ large,   because, 
  for all $s$ in the integration range $\bar I$, it holds that 
  $A_2 t>t^{\frac{1}{3}}$ and $A_1 t>t^{\frac{1}{3}}$.  
  Using this and the fact that 
  \be\Eq(schluss.1)
  \tilde m(t)^2/t=2t-\log t+O ( \log (t)^2/t),
  \ee
we see that  the expression in  \eqv(t1.no.40) is smaller than 
 \be\Eq(t1.no.41)
2 \wh C\int_0^1 
 \frac{ {t^{1-\frac{A_1}{2}}} }{A_2t\sqrt {A_1 t}}
  \eee^{-s+A_1t +O(s^\g)}\frac{dh}{(2\pi)^{\frac{3}{2}}},
 \ee
  Since  $A_1=1-A_2$, the fraction in   
  \eqv(t1.no.41) is bounded by a constant times
  \be\Eq(t1.no.45)
     \frac{t^{-1+A_2/{2}} }{A_2(1-A_2)^{\frac{1}{2}}}.
     \ee
  We now distinguish three regimes. If $A_2\in (\e,1-\e)$, for $\e>0$ independent of $t$,
  then the expression in \eqv(t1.no.45) is of order $t^{-1/2}$, as  $t\uparrow \infty$.
  If $A_2$ tends to $1$, 
  then for $s\in \bar I$, 
   \be\Eq(t1.no.45)
     \frac{t^{-1+A_2/{2}} }{A_2(1-A_2)^{\frac{1}{2}}}\leq t^{-1/2+1/3},
     \ee
which tends to zero, as $t\uparrow \infty$.
   Finally, when $A_2\downarrow 0$,  we get 
     \be\Eq(t1.no.47)
     \frac{t^{-1+A_2/{2}} }{A_2(1-A_2)^{\frac{1}{2}}}
    \leq t^{-1+2/3+o(1)},   \ee
  which tends to zero as
      $t\uparrow \infty$.
  Hence, for all $s\in \bar I$,   \eqv(t1.no.41) is bounded from above  by
 \be\Eq(t1.12)
  o(1)\int_0^1\eee^{-s+A_1t +O(s^\g)}dh.
 \ee
 Inserting this into \eqv(t1.no.3),  and writing out 
 $A_1t=h\S^2(s)+(1-h)\overline{\S}^2(s)$, we see that
 \bea\Eq(t1.no.50)
  &&\left |\E_n\left(\int_0^1(\overline{T1})dh\right)\right|
  \leq o(1) \int_{\bar I}|\Sigma^2(s)-\overline{\Sigma}^2(s)|\int_0^1\eee^{-s+(h\S^2(s)+(1-h)\overline{\S}^2(s)) +O(s^\g)} dh ds
 \nonumber\\
&&\leq 
 o(1)\int_{\bar I}\left|\eee^{-s+ \Sigma^2(s)+O(s^\g)} -
  \eee^{-s+ \overline{\Sigma}^2(s)+O(s^\g)}\right| ds,
 \eea
with  $o(1)$ tending to $0$, as  $t\uparrow \infty$, uniformly  for  $\k^2$ small enough.
This proves Lemma \thv(lem.t1).
 \end{proof}

 \begin{proof}[Proof of Lemma \thv(lem.t3)]
  We split the domain of integraion  into three  parts. First, let $\d_3>0$ be such that 
  \be\Eq(t3.1)
  \s_b^2+\sfrac{K}{2}\d_3<1 \mbox{ and }\d_3<\d_b.
  \ee
  By a Taylor expansion at zero we have
  \be\Eq(t3.2)
  \Sigma^2(s)\leq (\s_b^2+\sfrac{K}{2}\d_3)s,\quad \mbox{for }s\in[0,\d_3t].
  \ee
  Moreover, if $\d_1>0$, then so is $\d_0$, and we then choose $\d_3< 
  \d_0^<\land \d_1^<$ (with $\d^<_i\equiv \lim_{t\uparrow \infty}\d^<_i$); hence, for $t$ large enough it then also holds that 
  $\d_3< \d_0^<(t)\land \d_1^<(t)$.

 If $\d^<_1=0$, we set  (note that, by monotonicity, in this case $\d_0^<(t)\land \d_1^<(t)=\d_0^<(t)$)
  \bea\Eq(t3.3)
 (S1) &&\equiv \int_{t\d_0^<(t)}^{\d_3t }\left|\eee^{-s+ \Sigma^2(s)+O(s^\g)} - \eee^{-s+ \overline{\Sigma}^2(s)+O(s^\g)}\right|ds\nonumber\\
 && \leq \int_{t\d_0^<(t)}^{\d_3t }\left(\eee^{-s(1-\s_b^2-\frac{K}{2}\d_3)+O(s^\g)}+\eee^{-s(1-\s_b^2-\frac{K}{2}\d^<(t))+O(s^\g)}\right)ds.
  \eea
  By assumption on $\d_3$, $1-\s_b^2-\sfrac{K}{2}\d_3>0$ and $1-\s_b^2-\frac{K}{2}\d^<(t)>0$, for all t sufficiently large. Hence 
  \be\Eq(t3.4)
  \lim_{t\to\infty} (S1) =0.
  \ee
  If $\d_1^<>0$, we set $(S1)=0$. 
  
  Next we choose $\d_4$ such that
  \be\Eq(t3.5)
  \s_e^2-\d_4\sfrac{K}{2}>1\mbox{ and } \d_4<\d_e.
  \ee
  Again due to a first order Taylor expansion we have 
  \be\Eq(t3.6)
  \Sigma^2(t-\bar s)\leq t-\left(\s_e^2-\sfrac{K}{2}\d_4\right)\bar s,\quad \mbox{for }\bar s\in[t\d_1^>(t),t\d_4]
  \ee
  Hence
  \bea\Eq(t3.7)
  (S2) &&\equiv\int^{t(1-\d_1^>(t))}_{t-\d_4t }\left|\eee^{-s+ \Sigma^2(s)+O(s^\g)} - \eee^{-s+ \overline{\Sigma}^2(s)+O(s^\g)}\right|ds\nonumber\\
  &&= \int_{t\d_1^>(t)}^{\d_4t}\left|\eee^{\bar s- t+\Sigma^2(t-\bar s)+O(s^\g)}-\eee^{\bar s- t+\overline{\Sigma}^2(t-\bar s)+O(s^\g)}\right|d\bar s\nonumber\\
  &&\leq \int_{t\d_1^>(t)}^{\d_4t}\left(\eee^{\bar s(1-\s_e^2+\frac{K}{2}\d_4)+O(s^\g)}+\eee^{\bar s(1-\s_e^2+\frac{K}{2}\d^>(t))+O(s^\g)}\right)d\bar s.
  \eea
  By assumption on $\d_4$ we have $1-\k^2_e+\sfrac{K}{2}\d_4<0$ and, for $t$ large, $1-\s_e^2+\frac{K}{2}\d^>(t)<0$. Hence
  \be\Eq(t3.8)
  \lim_{t\to\infty}(S2)=0.
  \ee
  We still have to control
  \be\Eq(t3.9)
  (S3) \equiv\int^{t-\d_4t}_{\d_3t }\left|\eee^{-s+ \Sigma^2(s)+O(s^\g)} - \eee^{-s+ \overline{\Sigma}^2(s)+O(s^\g)}\right|ds.
  \ee
  Consider the function $A(x) $ on the interval $[\d_3,1-\d_4]$. Since $A(x)$ is right-continuous, increasing and $A(x)<x$ on $(0,1)$, we know that 
  \be\Eq(t3.10)
  M\equiv \inf_{x\in[\d_3,1-\d_4]}\left(x-A(x)\right) >0.
  \ee
  Then 
  \be\Eq(t3.11)
  s-\Sigma^2(s)=t(s/t-A(s/t))\geq Mt,
  \ee
  which implies
  \be\Eq(t3.12)
  \int_{\d_3t}^{t-\d_4t}\eee^{-s+\Sigma^2(s)+O(s^\g)}ds
  \leq \eee^{-Mt}\int_{\d_3t}^{t-\d_4t}\eee^{O(s^\g)}ds, 
  \ee
 which tends to zero,  as $t\uparrow \infty$. By the same argument it follows that
  \be\Eq(t3.13)
  \lim_{t\uparrow \infty} \int_{\d_3t}^{t-\d_4t}\eee^{-s+ \overline{\Sigma}^2(s)+O(s^\g)} ds =0.
  \ee
 It follows that  $\lim_{t\uparrow \infty}(S3)=0$, which concludes the proof of Lemma \thv(lem.t3). 
 \end{proof}


\end{document}